\documentclass{amsart}
\usepackage{amscd,amsmath,amssymb}
\usepackage{pstricks}
\usepackage{latexsym,amsbsy,mathrsfs}
\usepackage{xy}
\usepackage{xypic}
\usepackage{pgf,tikz}

\newtheorem{thm}{Theorem}[section]
\newtheorem{lem}[thm]{Lemma}
\newtheorem{cor}[thm]{Corollary}
\newtheorem{prop}[thm]{Proposition}

\theoremstyle{definition}
\newtheorem{example}[thm]{Example}

\newtheorem{defn}[thm]{Definition}

\newtheorem{rem}[thm]{Remark}

\numberwithin{equation}{thm}

\begin{document}
\title[ A general construction of $n$-ANGULATED CATEGORIES]
{A general construction of $n$-angulated categories using periodic injective resolutions  }

\author{Zengqiang Lin}
\address{ School of Mathematical sciences, Huaqiao University,
Quanzhou\quad 362021,  China.} \email{zqlin@hqu.edu.cn}

\thanks{This work was supported  by the Science Foundation of Huaqiao University (Grant No. 2014KJTD14)}



\begin{abstract}
Let $\mathcal{C}$ be an additive category equipped with an automorphism $\Sigma$.
We show how to obtain $n$-angulations of $(\mathcal{C},\Sigma)$ using some particular periodic injective resolutions. We give necessary and sufficient conditions on $(\mathcal{C},\Sigma)$ admitting  an $n$-angulation. Then we apply these characterizations to explain the standard construction of $n$-angulated categories and the $n$-angulated categories arising from some local rings. Moreover, we obtain a  class of new examples of $n$-angulated categories from quasi-periodic selfinjective algebras.
\end{abstract}

\maketitle

\section{Introduction}

Let $n$ be an integer greater than or equal to three. Motivated by the development of higher cluster tilting subcategories and the higher Auslander Reiten theory \cite{[Iy],[IY]}, Geiss, Keller and Oppermann  introduced the notion of $n$-angulated categories, which are $``$higher dimensional" analogues of triangulated categories, and gave a ``standard construction" of $n$-angulated categories from  $(n-2)$-cluster tilting subcategories of triangulated categories which are closed under the $(n-2)$-nd power of the suspension functor \cite{[GKO]}. For $n=3$, an $n$-angulated category is nothing but a classical triangulated category. Other examples of $n$-angulated categories arising from local rings were given in \cite{[BJT]}. The theory of $n$-angulated categories has been developed further, see \cite{[ABT],[BT1],[BT3],[J],[Jo],[L1],[L2],[L3],[LZ]} for example.

The first aim and motivation of this paper is to present a general framework to unify the two constructions of $n$-angulated categories in \cite{[GKO]} and \cite{[BJT]}. The second motivation is to construct new examples of $n$-angulated categories.

Roughly speaking, an {\em $n$-angulated category} is an additive category $\mathcal{C}$ equipped with an automorphism $\Sigma$ of $\mathcal{C}$ and a class $\Theta$ of $n$-$\Sigma$-sequences satisfying four axioms, denoted (N1)-(N4) (see 2.1 for precise definition). In this case, $\Theta$ is called an {\em $n$-angulation} of $(\mathcal{C},\Sigma)$. If $\Theta$  satisfies (N1)-(N3), then $\Theta$ is called a {\em pre-$n$-angulation} and its elements are called {\em $n$-angles}.

We first note that if $\mathcal{C}$ is an $n$-angulated category, then $\mbox{mod}\,\mathcal{C}$, the category of finitely presented  functors from $\mathcal{C}^{op}$ to the category Ab of abelian groups, is a Frobenius category \cite[Proposition 2.5(b)]{[GKO]}. An important property of $n$-angles is that all $n$-angles are exact \cite[Proposition 2.5(a)]{[GKO]}.
By Yoneda embedding $\iota: \mathcal{C}\rightarrow \mbox{proj}\,\mathcal{C}$, where $\mbox{proj}\,\mathcal{C}$ is the subcategory of $\mbox{mod}\,\mathcal{C}$ consists of projective objects, we identify $\mathcal{C}$ with $\iota(\mathcal{C})$, the essential image of $\iota$. Thus, an $n$-angle can be seen as an $n$-$\Sigma$-periodic  exact complex over $\iota(\mathcal{C})$. Furthermore, an $n$-angulation $\Theta$ can be identified as a full subcategory of $C^{\tiny\mbox{ex}}_{n\mbox{-}\Sigma}(\iota(\mathcal{C}))$, the category of $n$-$\Sigma$-periodic exact complexes over  $\iota(\mathcal{C})$.
According to this point of view,  an $n$-angulation turns out to be closed under translation functor and $n$-$\Sigma$-homotopy equivalence. Moreover, an $n$-angulation must contain all $n$-$\Sigma$-contractible complexes in $C^{\tiny\mbox{ex}}_{n\mbox{-}\Sigma}(\iota(\mathcal{C}))$.

Now assume that $\mathcal{C}$ is an additive category equipped with an automorphism $\Sigma$ and  $\mbox{mod}\,\mathcal{C}$ is a Frobenius category. Let $\Theta$ be a full subcategory of $C^{\tiny\mbox{ex}}_{n\mbox{-}\Sigma}(\iota(\mathcal{C}))$ which is closed under translation functor and $n$-$\Sigma$-homotopy equivalence. We construct a functor $Z_1: \Theta\rightarrow \mbox{mod}\,\mathcal{C}$ by sending a complex $X_\bullet=(X_i, f_i)_{i\in\mathbb{Z}}$ to $ M=\mbox{ker}f_1$. In this case, $X_\bullet$ is an $n$-$\Sigma$-periodic injective resolution of $M$. Our main theorem is the following.

\begin{thm}(cf. Theorem \ref{thm2}) The following statements are equivalent.

(a) The class $\Theta$ is an $n$-angulation of $(\mathcal{C},\Sigma)$.

(b) The functor $Z_1: \Theta\rightarrow \mbox{mod}\,\mathcal{C}$ is dense and $\Theta$ satisfies axiom (N4).

(c) The functor $Z_1: \Theta\rightarrow \mbox{mod}\,\mathcal{C}$ is dense and ``strongly" full.
\end{thm}

By definition, the functor $Z_1: \Theta\rightarrow \mbox{mod}\,\mathcal{C}$ is called {\em``strongly" full} if for each morphism $h: Z_1X_\bullet\rightarrow Z_1Y_\bullet$ in $\mbox{mod}\,\mathcal{C}$, where $X_\bullet,Y_\bullet\in\Theta$, there exists a morphism $T(h): X_\bullet\rightarrow Y_\bullet$ in $\Theta$ such that $Z_1(T(h))=h$ and the mapping cone $C(T(h))$ belongs to $\Theta$.

Actually, (a) $\Leftrightarrow$ (b) and (b) $\Rightarrow$ (c) come from a more general result, characterizing a pre-$n$-angulation in terms of the functor $Z_1$.

\begin{thm}(cf. Theorem \ref{thm1}) The following statements are equivalent.

(a) The class $\Theta$ is a pre-$n$-angulation of $(\mathcal{C},\Sigma)$.

(b) The functor $Z_1: \Theta\rightarrow \mbox{mod}\,\mathcal{C}$ is dense and full.
\end{thm}

Therefore, we can construct an $n$-angulation as follows. For each $M\in\mbox{mod}\,\mathcal{C}$, we fix an $n$-$\Sigma$-periodic injective resolution $T_M\in C^{\tiny\mbox{ex}}_{n\mbox{-}\Sigma}(\iota(\mathcal{C}))$. Denote by $\Theta$ the full subcategory  of $ C^{\tiny\mbox{ex}}_{n\mbox{-}\Sigma}(\iota(\mathcal{C}))$ consisting of objects $X_\bullet$ which is $n$-$\Sigma$-homotopy equivalent to some $T_M$. Actually, we can fix a minimal $n$-$\Sigma$-periodic injective resolution $T_M\in C^{\tiny\mbox{ex}}_{n\mbox{-}\Sigma}(\iota(\mathcal{C}))$ for each $M\in\mbox{mod}\,\mathcal{C}$. Then $\Theta$ is the full subcategory of $ C^{\tiny\mbox{ex}}_{n\mbox{-}\Sigma}(\iota(\mathcal{C}))$ consisting of objects isomorphic to $T_M\oplus A_\bullet$ where  $A_\bullet$ is an $n$-$\Sigma$-contractible complex. If $\Theta$ is not closed under the translation functor, then $\Theta$ is not an $n$-angulation of $(\mathcal{C},\Sigma)$. Otherwise, $\Theta$ is an  $n$-angulation if $\Theta$ moreover satisfies (N4). Furthermore, each $n$-angulation can be constructed in this way.

Recall that given a field $k$, a finite-dimensional $k$-algebra $A$ is said to be $quasi$-$periodic$ if  $A$ has a quasi-periodic projective resolution over the enveloping algebra $A^e=A^{\tiny\mbox{op}}\otimes_k A$, i.e., $\Omega^n_{A^e}(A)\cong$ $_1A_\sigma$ as $A$-$A$-bimodule for some natural number $n$ and some automorphism $\sigma$ of $A$. In particular, $A$ is  $periodic$ if $\Omega^n_{A^e}(A)\cong A$. Typical examples of periodic algebras are  preprojective algebras of Dynkin type, Brauer tree algebras, algebras of quaternion type, selfinjective algebra of finite representation type, some $d$-cluster tilted algebras, deformed mesh algebras of generalized Dynkin type and so on; see \cite{[BES]}, \cite{[D1]}, \cite{[D2]}, \cite{[ES]}.

It is well known that periodic algebras are selfinjective and their module categories are periodic, i.e., each finite dimensional module has a periodic injective resolution. More generally, assume that $A$ is a finite-dimensional indecomposable $k$-algebra such that $\Omega^n_{A^e}(A)\cong$ $_1A_\sigma$ as an $A$-$A$-bimodule for some automorphism $\sigma$ of $A$ and for some $n\geq3$. Then $A$ is a selfinjective algebra and the functor $-\otimes_A(_1A_{\sigma^{-1}}): \mbox{proj}A\rightarrow\mbox{proj}A$ is an automorphism. As an application of Theorem 1.1, we can construct a new class of examples of $n$-angulated categories from quasi-periodic selfinjective algebras.

\begin{thm}(cf. Theorem \ref{4.2})
 The category $(\mbox{proj}\,A, -\otimes_A(_1A_{\sigma^{-1}}))$ admits an $n$-angulation.
\end{thm}

This paper is organized as follows.
In Section 2, we recall the definition of $n$-angulated category and prove some lemmas.

In Section 3, we show how to obtain $n$-angulations of $(\mathcal{C},\Sigma)$ using $n$-$\Sigma$-periodic injective resolutions. We first develop some needed properties on $n$-$\Sigma$-periodic injective resolution to provide some new views on $n$-angles. Then we prove Theorem \ref{thm1} and Theorem \ref{thm2}. As applications, we get Corollary \ref{cor1} and Corollary \ref{cor2}, which is a higher version of a result of Amiot \cite[Theorem 8.1]{[Am]}.

The last two sections are devoted to demonstrate that our approach is general, both for known examples and new examples.

In Section 4,  we give a new point of view on known examples including  algebraic triangulated categories, the standard construction of $n$-angulated categories given in \cite{[GKO]} and the $n$-angulated categories given in \cite{[BJT]}. Therefore, our ideas indeed present a general framework to unify the known constructions of $n$-angulated categories.

 In Section 5, we provide some new examples of $n$-angulated categories. We first discuss the $n$-angulated structure on semisimple categories and give some new characterizations of semisimple categories; see Theorem \ref{thm0} and Corollary \ref{cor3}. Then we apply  Corollary \ref{cor2} to obtain a large class of new examples of $n$-angulated categories from quasi-periodic selfinjective algebras; see Theorem \ref{4.2}. To close this section, we construct other $n$-angulations from known ones. We show that the group of global automorphisms of $(\mathcal{C},\Sigma)$ acts on the set of $n$-angulations of $(\mathcal{C},\Sigma)$ from the right; see Proposition \ref{prop}.

\section{Definitions and preliminaries}
Throughout this paper, we assume that $\mathcal{C}$ is an additive category  equipped with an automorphism $\Sigma:\mathcal{C}\rightarrow\mathcal{C}$.
We denote by  $\mathcal{C}(X,Y)$  the set of morphisms from $X$ to $Y$ in $\mathcal{C}$. We denote the composition of $f\in \mathcal{C}(X,Y)$ and $g\in \mathcal{C}(Y,Z)$ by $gf\in \mathcal{C}(X,Z)$. We denote by $\mbox{mod}\,\mathcal{C}$  the category of finitely presented functors from $\mathcal{C}^{op}$ to the category Ab of abelian groups and by $\mbox{proj}\,\mathcal{C}$ the subcategory of $\mbox{mod}\,\mathcal{C}$ which consists of projective objects. Given a functor $F:\mathcal{C}\rightarrow\mathcal{D}$, we denote by $F(\mathcal{C})$ the essential image of $F$.

\vspace{2mm}
 An $n$-$\Sigma$-$sequence$ in $\mathcal{C}$ is a sequence of morphisms
$$X_\bullet= (X_1\xrightarrow{f_1}X_2\xrightarrow{f_2}X_3\xrightarrow{f_3}\cdots\xrightarrow{f_{n-1}}X_n\xrightarrow{f_n}\Sigma X_1).$$
Its $left\ rotation$ is the $n$-$\Sigma$-sequence
$$X_2\xrightarrow{f_2}X_3\xrightarrow{f_3}X_4\xrightarrow{f_4}\cdots\xrightarrow{f_{n-1}}X_n\xrightarrow{f_n}\Sigma X_1\xrightarrow{(-1)^n\Sigma f_1}\Sigma X_2.$$  We can define $right\ rotation$ of an $n$-$\Sigma$-sequence similarly.
An $n$-$\Sigma$-sequence $X_\bullet$ is $exact$ if the induced sequence
$$\cdots\rightarrow \mathcal{C}(-,X_1)\rightarrow \mathcal{C}(-,X_2)\rightarrow\cdots\rightarrow \mathcal{C}(-,X_n)\rightarrow \mathcal{C}(-,\Sigma X_1)\rightarrow\cdots$$
is exact. A $trivial$ $n$-$\Sigma$-sequence is a sequence of the form $$X\xrightarrow{1}X\rightarrow 0\rightarrow\cdots\rightarrow 0\rightarrow \Sigma X$$
or any of its rotations.
A $morphism$ of $n$-$\Sigma$-sequences is  a sequence of morphisms $\varphi_\bullet=(\varphi_1,\varphi_2,\cdots,\varphi_n)$ such that the following diagram
$$\xymatrix{
X_1 \ar[r]^{f_1}\ar[d]^{\varphi_1} & X_2 \ar[r]^{f_2}\ar[d]^{\varphi_2} & X_3 \ar[r]^{f_3}\ar[d]^{\varphi_3} & \cdots \ar[r]^{f_{n-1}}& X_n \ar[r]^{f_n}\ar[d]^{\varphi_n} & \Sigma X_1 \ar[d]^{\Sigma \varphi_1}\\
Y_1 \ar[r]^{g_1} & Y_2 \ar[r]^{g_2} & Y_3 \ar[r]^{g_3} & \cdots \ar[r]^{g_{n-1}} & Y_n \ar[r]^{g_n}& \Sigma Y_1\\
}$$
commutes where each row is an $n$-$\Sigma$-sequence. It is an {\em isomorphism} if $\varphi_1, \varphi_2, \cdots, \varphi_n$ are all isomorphisms in $\mathcal{C}$.

\begin{defn} (\cite{[GKO]}) Let $\mathcal{C}$ be an additive category and $\Sigma$ an automorphism of $\mathcal{C}$. A collection $\Theta$ of $n$-$\Sigma$-sequences is called a $pre$-$n$-$angulation$ and its elements are called $n$-$angles$ if $\Theta$ satisfies the following axioms:

(N1) (a) The class $\Theta$ is closed under isomorphisms, direct sums and direct summands.

(b) For each object $X\in\mathcal{C}$ the trivial sequence
$$X\xrightarrow{1}X\rightarrow 0\rightarrow\cdots\rightarrow 0\rightarrow \Sigma X$$
belongs to $\Theta$.

(c) For each morphism $f_1:X_1\rightarrow X_2$ in $\mathcal{C}$, there exists an $n$-$\Sigma$-sequence in $\Theta$ whose first morphism is $f_1$.

(N2) An $n$-$\Sigma$-sequence belongs to $\Theta$ if and only if its left rotation belongs to $\Theta$.

(N3) Each commutative diagram
$$\xymatrix{
X_1 \ar[r]^{f_1}\ar[d]^{\varphi_1} & X_2 \ar[r]^{f_2}\ar[d]^{\varphi_2} & X_3 \ar[r]^{f_3}\ar@{-->}[d]^{\varphi_3} & \cdots \ar[r]^{f_{n-1}}& X_n \ar[r]^{f_n}\ar@{-->}[d]^{\varphi_n} & \Sigma X_1 \ar[d]^{\Sigma \varphi_1}\\
Y_1 \ar[r]^{g_1} & Y_2 \ar[r]^{g_2} & Y_3 \ar[r]^{g_3} & \cdots \ar[r]^{g_{n-1}} & Y_n \ar[r]^{g_n}& \Sigma Y_1\\
}$$ with rows in $\Theta$ can be completed to a morphism of  $n$-$\Sigma$-sequences.

\vspace{2mm}
In this case, we say $(\mathcal{C},\Sigma,\Theta)$ is a $pre$-$n$-$angulated$ $category$. If $\Theta$ moreover satisfies the following axiom, then it is called an $n$-$angulation$ of $(\mathcal{C},\Sigma)$ and  $(\mathcal{C},\Sigma,\Theta)$ is called an $n$-$angulated$ $category$:

\vspace{2mm}
(N4) In the situation of (N3), the morphisms $\varphi_3, \cdots,\varphi_n$ can be chosen such that the mapping cone
$$X_2\oplus Y_1\xrightarrow{\left(
                              \begin{smallmatrix}
                                -f_2 & 0 \\
                                \varphi_2 & g_1 \\
                              \end{smallmatrix}
                            \right)}
 X_3\oplus Y_2 \xrightarrow{\left(
                              \begin{smallmatrix}
                                -f_3 & 0 \\
                                \varphi_3 & g_2 \\
                              \end{smallmatrix}
                            \right)}
 \cdots \xrightarrow{\left(
                            \begin{smallmatrix}
                               -f_n & 0 \\
                                \varphi_n & g_{n-1} \\
                             \end{smallmatrix}
                           \right)}
 \Sigma X_1\oplus Y_n \xrightarrow{\left(
                              \begin{smallmatrix}
                                -\Sigma f_1 & 0 \\
                                \Sigma\varphi_1 & g_n \\
                              \end{smallmatrix}
                            \right)}
 \Sigma X_2\oplus \Sigma Y_1 \\
$$
belongs to $\Theta$.
\end{defn}

In the rest of this section, we will collect some facts on pre-$n$-angulated categories.

\begin{lem}(\cite[Proposition 2.5]{[GKO]})\label{1.1}
Let $(\mathcal{C}, \Sigma, \Theta)$ be a pre-$n$-angulated category. Then the following hold.

(a) All $n$-angles in $\mathcal{C}$ are exact.

(b) The category $\mbox{mod}\,\mathcal{C}$ is an abelian Frobenius category.
\end{lem}

\begin{lem}\label{lem1.1}
Let $X_1\xrightarrow{f_1}X_2\xrightarrow{f_2}\cdots\xrightarrow{f_{n-1}}X_{n}\xrightarrow{f_n}\Sigma X_1$ be an $n$-angle. Then the following statements are equivalent.

(a) $f_1$ is a split monomorphism;

(b) $f_{n-1}$ is a split epimorphism;

(c) $f_n=0$.
\end{lem}

\begin{proof}
We only prove (a) $\Leftrightarrow$ (c), since we can prove (b) $\Leftrightarrow$ (c) similarly. If  $f_1$ is a split monomorphism, then there exists a morphism $f_1':X_2\rightarrow X_1$ such that $f_1'f_1=1$. By (N3), we obtain the following commutative diagram
$$\xymatrix{
X_1 \ar[r]^{f_1}\ar@{=}[d] & X_2 \ar[r]^{f_2}\ar[d]^{f'_1} & X_3 \ar[r]^{f_3}\ar@{-->}[d] & \cdots \ar[r]^{f_{n-1}}& X_n \ar[r]^{f_n}\ar@{-->}[d] & \Sigma X_1 \ar@{=}[d]\\
X_1 \ar[r]^{1} & X_1 \ar[r] & 0 \ar[r] & \cdots \ar[r] & 0 \ar[r]& \Sigma X_1\\
}$$
whose rows are $n$-angles. Thus $f_n=0$. Conversely, if  $f_n=0$, then the diagram $$\xymatrix{
X_1 \ar[r]^{f_1}\ar@{=}[d] & X_2 \ar[r]^{f_2} & X_3 \ar[r]^{f_3} & \cdots \ar[r]^{f_{n-1}}& X_n \ar[r]^{f_n}\ar[d] & \Sigma X_1 \ar@{=}[d]\\
X_1 \ar[r]^{1} & X_1 \ar[r] & 0 \ar[r] & \cdots \ar[r] & 0 \ar[r]& \Sigma X_1\\
}$$ can be completed as a morphism of $n$-angles by (N2) and (N3), which implies that $f_1$ is a split monomorphism.
\end{proof}

\begin{lem}\label{lem1.2}
Let $f:X\rightarrow Y$ be a morphism in a pre-$n$-angulated category $\mathcal{C}$.

(a) If $f$ is a monomorphism, then $f$ is a split monomorphism;

(b)  If $f$ is an epimorphism, then $f$ is a split epimorphism.
\end{lem}

\begin{proof}
(a) By (N1)(c), we assume that $X\xrightarrow{f}Y\xrightarrow{f_2}X_3\xrightarrow{f_3}\cdots\xrightarrow{f_{n-1}}X_{n}\xrightarrow{f_n}\Sigma X$ is an $n$-angle. Since $f\cdot\Sigma^{-1}f_n=0$ and $f$ is a monomorphism, we have $\Sigma^{-1}f_n=0$. Thus $f_n=0$. By Lemma \ref{lem1.1},  $f$ is a split monomorphism.
We can prove (b) dually.
\end{proof}

\section{Periodic injective resolutions and $n$-angulations}

In this section, we always assume that $\mbox{mod}\,\mathcal{C}$ is a Frobenius category, which may be implied in other condition; see Lemma \ref{lem2}(a). Thus the stable category $\underline{\mbox{mod}}\,\mathcal{C}$ is a triangulated category with the suspension functor $\Omega^{-1}$. The automorphism $\Sigma$ of $\mathcal{C}$ induces an automorphism $\Sigma$ of $\mbox{mod}\,\mathcal{C}$ by mapping  $M$ to $M\cdot \Sigma^{-1}$.
The Yoneda functor $\iota:\mathcal{C}\rightarrow \mbox{proj}\,\mathcal{C}$, which maps $A$ to $\mathcal{C}(-,A)$, gives a natural equivalence between $\mathcal{C}$ and $\iota(\mathcal{C})$. We note that $\iota(\mathcal{C})$ is not equal to $\mbox{proj}\,\mathcal{C}$ in general unless $\mathcal{C}$ is idempotent complete. For convenience, we identify $\mathcal{C}$ with $\iota(\mathcal{C})$.

\subsection{$n$-$\Sigma$-periodic injective resolutions}

An exact complex $X_\bullet=(X_i, f_i)_{i\in\mathbb{Z}}$ over $\iota(\mathcal{C})$ is called $n$-$\Sigma$-$periodic$ if $X_{k+n}=\Sigma X_k$ and $f_{k+n}=\Sigma f_k$ for all $k\in\mathbb{Z}$. 
Since we have identified $\mathcal{C}$ with $\iota(\mathcal{C})$, 
we can identify an exact $n$-$\Sigma$-sequence in $\mathcal{C}$ as an $n$-$\Sigma$-periodic exact complexe over $\iota(\mathcal{C})$, and vice versa.

A morphism $\varphi_\bullet=(\varphi_i)_{i\in\mathbb{Z}}$ between two $n$-$\Sigma$-periodic exact complexes $X_\bullet$ and $Y_\bullet$ is given by a morphism of complexes such that $\varphi_{k+n}=\Sigma\varphi_k$ for all $k\in\mathbb{Z}$. We denote by $C^{\tiny\mbox{ex}}_{n\mbox{-}\Sigma}(\iota(\mathcal{C}))$ the category of $n$-$\Sigma$-periodic exact complexes over $\iota(\mathcal{C})$.
 Let $X_\bullet=(X_i, f_i)_{i\in\mathbb{Z}}, Y_\bullet=(Y_i, g_i)_{i\in\mathbb{Z}}$ be two objects in $C^{\tiny\mbox{ex}}_{n\mbox{-}\Sigma}(\iota(\mathcal{C}))$, and $\varphi_\bullet,\psi_\bullet$ be two morphisms from $X_\bullet$ to $Y_\bullet$. An $n$-$\Sigma$-$homotopy$ from $\varphi_\bullet$ to $\psi_\bullet$ is  given by morphisms $h_i: X_{i+1}\rightarrow Y_i$ such that $\varphi_i-\psi_i=h_if_i+g_{i-1}h_{i-1}$ and $h_{n+i}=\Sigma h_i$ for all $i\in\mathbb{Z}$. In this case, we say that $\varphi_\bullet$ and $\psi_\bullet$ are $n$-$\Sigma$-$homotopic$. Since $n$-$\Sigma$-homotopy relation is an equivalence relation, we can form the relative homotopy category $K^{\tiny\mbox{ex}}_{n\mbox{-}\Sigma}(\iota(\mathcal{C}))$, by considering the objects are the same as those of $C^{\tiny\mbox{ex}}_{n\mbox{-}\Sigma}(\iota(\mathcal{C}))$ and the additive group of $n$-$\Sigma$-homotopy classes of morphisms from $X_\bullet$ to $Y_\bullet$ in $C^{\tiny\mbox{ex}}_{n\mbox{-}\Sigma}(\iota(\mathcal{C}))$ as the group of morphisms from $X_\bullet$ to $Y_\bullet$ in $K^{\tiny\mbox{ex}}_{n\mbox{-}\Sigma}(\iota(\mathcal{C}))$.

\begin{lem}
The relative homotopy category $K^{\tiny\mbox{ex}}_{n\mbox{-}\Sigma}(\iota(\mathcal{C}))$ is a triangulated category.
\end{lem}

\begin{proof}
Denote by $\mathcal{S}$ the class of all chain-wise split exact sequences $0\rightarrow X_\bullet\xrightarrow{\varphi_\bullet} Y_\bullet\xrightarrow{\psi_\bullet} Z_\bullet\rightarrow 0$ in $C^{\tiny\mbox{ex}}_{n\mbox{-}\Sigma}(\iota(\mathcal{C}))$. Then $(C^{\tiny\mbox{ex}}_{n\mbox{-}\Sigma}(\iota(\mathcal{C})),\mathcal{S})$ is an exact category. Similar to \cite[Proposition 7.1]{[PX]}, we can show that $(C^{\tiny\mbox{ex}}_{n\mbox{-}\Sigma}(\iota(\mathcal{C})),\mathcal{S})$ is a Frobenius category and the projective-injectives are the $n$-$\Sigma$-contractible complexes, whose identity morphisms are $n$-$\Sigma$-homotopic to the zero morphisms.  Thus  the stable category $\underline{C^{\tiny\mbox{ex}}_{n\mbox{-}\Sigma}(\iota(\mathcal{C}))}=K^{\tiny\mbox{ex}}_{n\mbox{-}\Sigma}(\iota(\mathcal{C}))$ is a triangulated category  whose suspension functor is the translation functor.
\end{proof}

\begin{lem}\label{lem1}
 The following hold.

(a) The functor $Z_1: C^{\tiny\mbox{ex}}_{n\mbox{-}\Sigma}(\iota(\mathcal{C}))\rightarrow \mbox{mod}\,\mathcal{C}$,  which sends a complex $X_\bullet=(X_i, f_i)_{i\in\mathbb{Z}}$ to ker$f_1$, induces a triangle functor $\underline{Z_1}: K^{\tiny\mbox{ex}}_{n\mbox{-}\Sigma}(\iota(\mathcal{C}))\rightarrow \underline{\mbox{mod}}\,\mathcal{C}$.

(b) If $\underline{Z_1}(\varphi_\bullet)=0$, then $\varphi_\bullet$ is $n$-$\Sigma$-homotopic to some morphism $\varphi_\bullet'=(0,0,\cdots,0,\varphi_n')$.

(c) The kernel of $\underline{Z_1}$ is an ideal whose square vanishes.

\end{lem}

\begin{proof}
(a) By the Snake Lemma, we can show that $Z_1$ is an exact functor. Then $Z_1$ induces a triangle functor $\underline{Z_1}$ since $Z_1$ preserves the projective-injectives.

 (b) Let $k:M\rightarrow X_1$ be the kernel of $f_1$ and $k':N\rightarrow Y_1$ be the kernel of $g_1$. We assume that $\Sigma^{-1}f_n=k\pi$ and $\Sigma^{-1}g_n=k'\pi'$. Since $\underline{Z_1}(\varphi_\bullet)=0$, the morphism $h=Z_1(\varphi_\bullet)$ admits a factorization $M\xrightarrow{a}I\xrightarrow{b}N$, where $I$ is projective-injective. Thus there exist two morphisms $c:X_1\rightarrow I$ and $d:I\rightarrow\Sigma^{-1}Y_n$ such that $a=ck$ and $b=\pi'd$. We put $h'_n=\Sigma (dc)$.
 Now we assume that $f_i:X_i\rightarrow X_{i+1}$ has a factorization $k_i\pi_i: X_i\twoheadrightarrow M_i\rightarrowtail X_{i+1}$ for $i=1,2,\cdots,n-1$.
 Note that $(\varphi_1-\Sigma^{-1}(g_nh_n'))\Sigma^{-1}f_n=0$, hence there exists a morphism $m_1:M_1\rightarrow Y_1$ such that $\varphi_1-\Sigma^{-1}(g_nh_n')=m_1\pi_1$. Since $Y_1$ is projective-injective, there exists a morphism $h_1:X_2\rightarrow Y_1$ such that $m_1=h_1k_1$. Thus $\varphi_1=\Sigma^{-1}(g_nh_n')+h_1f_1$.
Similarly there exist morphisms $h_{i}:X_{i+1}\rightarrow Y_{i}$ for $2\leq i\leq n$ such that $\varphi_i=h_{i}f_i+g_{i-1}h_{i-1}$.
Suppose that $\varphi_n'=(h_{n}-h_n')f_n$, then $\varphi_n-\varphi_n'=g_{n-1}h_{n-1}+h_n' f_n$.
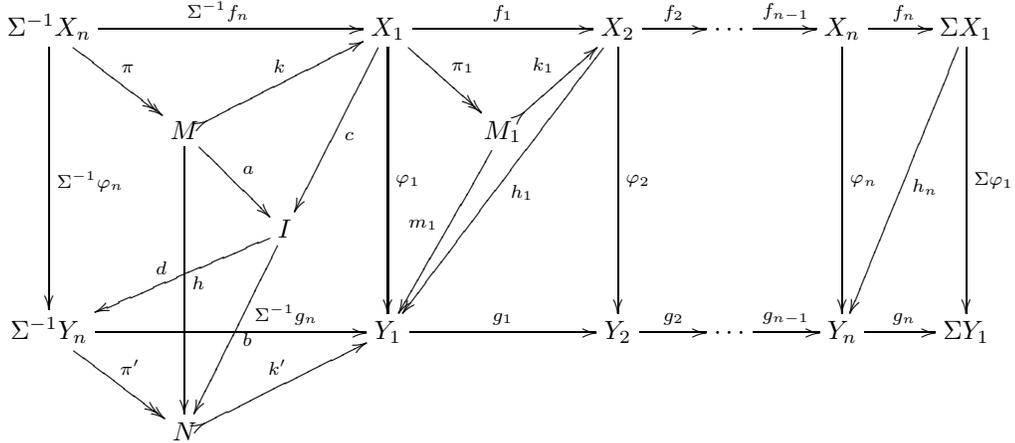
$$\xymatrix{
\Sigma^{-1}X_n \ar[ddd]^{\Sigma^{-1}\varphi_n} \ar[rrr]^{\Sigma^{-1}f_n} \ar@{->>}[rd]^{\pi}  &  &  &  X_1 \ar[rr]^{f_1} \ar[ddd]^{\varphi_1}\ar[ldd]^{c}\ar@{->>}[rd]^{\pi_1} & & X_2 \ar[ddd]^{\varphi_2}\ar[llddd]^{h_1} \ar[r]^{f_2}& \cdots \ar[r]^{f_{n-1}} & X_n \ar[r]^{f_n}\ar[ddd]^{\varphi_n} & \Sigma X_1\ar[ddd]^{\Sigma\varphi_1} \ar[lddd]^{h_{n}}\\
 & M \ar[ddd]^{h} \ar@{>->}[urr]^{k}\ar[rd]^a & & & M_1 \ar@{>->}[ru]^{k_1}\ar[ldd]_{m_1} &  & & & & \\
&  &  I \ar[ldd]^b \ar[lld]_{d} & & & & & &\\
\Sigma^{-1}Y_n \ar[rrr]^{\hspace{18mm}\Sigma^{-1}g_n}\ar@{->>}[rd]^{\pi'} & & & Y_1 \ar[rr]^{g_1} & & Y_2 \ar[r]^{g_2} & \cdots \ar[r]^{g_{n-1}} & Y_n \ar[r]^{g_n} &\Sigma Y_1\\
&  N \ar@{>->}[urr]^{k'}& & & & & & &\\
}$$
Hence the morphism $\varphi_\bullet$ is $n$-$\Sigma$-homotopic to the morphism $\varphi_\bullet'=(0,0,\cdots,0,\varphi_n')$ with an $n$-$\Sigma$-homotopy $(h_1, h_2, \cdots,h_{n-1}, h'_n)$.

(c)
Let $\varphi_\bullet:X_\bullet\rightarrow Y_\bullet$ and $\psi_\bullet: Y_\bullet\rightarrow Z_\bullet$ be two morphisms in the kernel of $\underline{Z_1}$. Up to $n$-$\Sigma$-homotopy, we assume that $\varphi_\bullet=(0,0,\cdots,0,\varphi_n)$ and $\psi_\bullet=(0,0,\cdots,0,\psi_n)$. Thus we get the following commutative diagram
$$\xymatrix{
X_1 \ar[r]^{f_1}\ar[d]^{0} & X_2 \ar[r]^{f_2}\ar[d]^{0} & \cdots \ar[r]^{f_{n-2}}& X_{n-2} \ar[r]^{f_{n-1}}\ar[d]^{0}& X_n \ar[r]^{f_n}\ar[d]^{\varphi_n}\ar@{-->}[ld]^{a_n} & \Sigma X_1 \ar[d]^{0}\\
Y_1 \ar[r]^{g_1}\ar[d]^0 & Y_2 \ar[r]^{g_2} \ar[d]^0 & \cdots \ar[r]^{g_{n-2}}& Y_{n-1} \ar[r]^{g_{n-1}}\ar[d]^0 & Y_n \ar[r]^{g_n}\ar[d]^{\psi_n}& \Sigma Y_1\ar[d]^0\ar@{-->}[ld]^{b_{n+1}}\\
Z_1\ar[r]^{h_1} & Z_2 \ar[r]^{h_2}  & \cdots \ar[r]^{h{n-2}} & Z_{n-1} \ar[r]^{h_{n-1}} & Z_n \ar[r]^{h_n} & \Sigma Z_1. \\
}$$
Since $g_n\varphi_n=0$ and $\psi_ng_{n-1}=0$, $\varphi_n$ factors through $g_{n-1}$ and $\psi_n$ factors through $g_n$. Thus $\psi_n\varphi_n=b_{n+1}g_ng_{n-1}a_n=0$. Therefore, $\psi_\bullet\varphi_\bullet=0$.
\end{proof}

\begin{lem}\label{lem2}

 Let  $\Theta$ be a full subcategory of $C^{\tiny\mbox{ex}}_{n\mbox{-}\Sigma}(\iota(\mathcal{C}))$ and $\varphi_\bullet: X_\bullet\rightarrow Y_\bullet$ be a morphism in $\Theta$.

(a) If the functor $Z_1: \Theta\rightarrow \mbox{mod}\,\mathcal{C}$ is full, then $\mbox{mod}\,\mathcal{C}$ is Frobenius.

(b) If the functor $Z_1: \Theta\rightarrow \mbox{mod}\,\mathcal{C}$ is full, then $Z_1(\varphi_\bullet)$ is an isomorphism in $\underline{\mbox{mod}}\,\mathcal{C}$ if and only if $\varphi_\bullet$ is an $n$-$\Sigma$-homotopy equivalence.

(c) If the functor $Z_1: \Theta\rightarrow \mbox{mod}\,\mathcal{C}$ is full and dense, then for each $M\in\mbox{mod}\,\mathcal{C}$, up to $n$-$\Sigma$-homotopy equivalence, there exists a unique $T_M\in\Theta$ such that $Z_1T_M\cong M$.
\end{lem}

\begin{proof}
(a) This is an adaptation of the proof of \cite[Propostion 2.5(b)]{[GKO]}. It is clear that $\mbox{mod}\,\mathcal{C}$ has enough projectives.  We claim that each projective $\mathcal{C}(-,A)$ is an injective. Indeed, for each $M\in\mbox{mod}\,\mathcal{C}$, we note that $\Sigma^{-1}M\in\mbox{mod}\,\mathcal{C}$. Since $Z_1$ is dense, there exists an object $X_\bullet\in \Theta$ such that $Z_1X_\bullet\cong\Sigma^{-1}M$. Thus $X_\bullet[-1]$ is an $n$-$\Sigma$-periodic projective resolution of $M$. For each $i\geq1$, by the Yoneda Lemma, it is not hard to see that Ext$^i(M, \mathcal{C}(-,A))\cong H^i(X_{\bullet}[-1],\mathcal{C}(-,A))=0$ since $X_\bullet[-1]$ is exact. Consequently, each projective is an injective. Since $Z_1$ is dense, it is easy to see that $\mbox{mod}\,\mathcal{C}$ has enough injectives, moreover, the projectives and injectives coincide. Therefore, $\mbox{mod}\,\mathcal{C}$ is Frobenius.

(b) The ``if" part is obvious. For the ``only if" part,  we only need to show that the functor $\underline{Z_1}:\Theta\rightarrow \underline{\mbox{mod}}\,\mathcal{C}$  detects isomorphisms. Since a full functor whose kernel is an ideal such that square vanishes detects isomorphisms \cite[Lemma 8.6]{[Am]}, the ``only if" part follows from Lemma \ref{lem1}(c).

(c) It is a direct consequence of (b).
\end{proof}


\begin{rem}
(a)  Let $\Theta_1\subseteq\Theta_2$ be full subcategories of $C^{\tiny\mbox{ex}}_{n\mbox{-}\Sigma}(\iota(\mathcal{C}))$ which are closed under $n$-$\Sigma$-homotopy equivalence. If the functors $Z_1: \Theta_1\rightarrow \mbox{mod}\,\mathcal{C}$ and $Z_1: \Theta_2\rightarrow \mbox{mod}\,\mathcal{C}$ are full and dense, then $\Theta_1=\Theta_2$. 

 (b) Assume that $X_\bullet\in C^{\tiny\mbox{ex}}_{n\mbox{-}\Sigma}(\iota(\mathcal{C}))$ and $M\in\mbox{mod}\,\mathcal{C}$. If $Z_1X_\bullet\cong M$, then $X_\bullet$ is an $n$-$\Sigma$-periodic injective resolution of $M$. Moreover, $X_\bullet$ is an $n$-$\Sigma$-periodic injective resolution of $M$ if and only if $X_\bullet[-1]$ is an $n$-$\Sigma$-periodic projective resolution of $\Sigma M$.
\end{rem}

\subsection{The class of pre-$n$-angulations}

Since  $\mbox{mod}\,\mathcal{C}$ is a Frobenius category, 
for each $M\in\mbox{mod}\,\mathcal{C}$, we fix a short exact sequence $0\rightarrow M\rightarrow I_M\rightarrow\Omega^{-1}M\rightarrow0$ with $I_M\in \mbox{proj}\,\mathcal{C}$. Thus we obtain a standard injective resolution of $M$ as follows:
$$\begin{gathered}I_M\rightarrow I_{\Omega^{-1}M}\rightarrow I_{\Omega^{-2}M}\rightarrow\cdots \end{gathered}\eqno(3.1)$$
Note that the automorphism $\Sigma$ of $\mbox{mod}\,\mathcal{C}$ is an exact functor. We may assume that $I_{\Sigma M}=\Sigma I_M$ and $\Sigma\Omega^{-1}M=\Omega^{-1}\Sigma M$. In fact, the automorphism $\Sigma$ of $\mathcal{C}$ induces a triangle functor $(\Sigma,\sigma)$ on $\underline{\mbox{mod}}\,\mathcal{C}$. It is well known that $(\Omega^{-n},(-1)^n1_{\Omega^{-n-1}})$ is also a triangle functor on $\underline{\mbox{mod}}\,\mathcal{C}$.

Suppose that there exists an isomorphism $\alpha:(\Sigma,\sigma) \rightarrow(\Omega^{-n},(-1)^n1_{\Omega^{-n-1}})$ of triangle functors. Let $X_\bullet=(X_1\xrightarrow{f_1}X_2\xrightarrow{f_2}X_3\xrightarrow{f_3}\cdots\xrightarrow{f_{n-1}}X_n\xrightarrow{f_n}\Sigma X_1)$
be an exact $n$-$\Sigma$-sequence in $\mathcal{C}$  and $M=\mbox{ker}f_1$. Note that $X_\bullet$ can be seen as an $n$-$\Sigma$-periodic injective resolution of $M$ and $f_n$ has a factorization $X_n\twoheadrightarrow \Sigma M\rightarrowtail \Sigma X_1$, so by the Comparison Theorem \cite[Theorem 6.16]{[Ro]}, there exists an isomorphism $\beta_M:\Sigma M\xrightarrow{\sim}\Omega^{-n}M$ in $\underline{\mbox{mod}}\,\mathcal{C}$. We denote by $\Theta_\alpha$ the class of exact $n$-$\Sigma$-sequences 
$X_\bullet$ in $\mathcal{C}$ such that $\beta_{M}=\alpha_{M}$, where $M=Z_1X_\bullet$.

\begin{lem}\label{2.3}
(\cite[Proposition 3.4]{[GKO]}) Let $(\mathcal{C},\Sigma,\Theta)$ be a pre-$n$-angulated category. Then
the map from the class of isomorphisms of triangle functors between $(\Sigma,\sigma)$ and $(\Omega^{-n},(-1)^n1_{\Omega^{-n-1}})$ to the class of pre-$n$-angulations of $(\mathcal{C},\Sigma)$, mapping $\alpha$ to $\Theta_\alpha$, is a bijection.
\end{lem}

It follows from (N2) that a pre-$n$-angulation is closed under translation functor.
The following lemma implies that a pre-$n$-angulation is closed under $n$-$\Sigma$-homotopy equivalence.

\begin{lem}\label{2.2} Suppose that  $(\mathcal{C},\Sigma,\Theta)$ is a pre-$n$-angulated category.
Let $X_\bullet$, $Y_\bullet$ be two objects in $C^{\tiny\mbox{ex}}_{n\mbox{-}\Sigma}(\iota(\mathcal{C}))$, and $\varphi_\bullet: X_\bullet\rightarrow Y_\bullet$ be an $n$-$\Sigma$-homotopy equivalence.
Then $X_\bullet\in\Theta$ if and only if $Y_\bullet\in\Theta$.
\end{lem}

\begin{proof} By Lemma \ref{2.3}, we can assume that $\Theta=\Theta_\alpha$ for some isomorphism $\alpha: (\Sigma,\sigma) \rightarrow (\Omega^{-n},(-1)^n1_{\Omega^{-n-1}})$ of triangle functors.
Assume that $M=Z_1X_\bullet$ and $N=Z_1Y_\bullet$, then $\varphi_\bullet$ induces a morphism $h=Z_1(\varphi_\bullet):M\rightarrow N$. We note that $h$ is an isomorphism in $\underline{\mbox{mod}}\,\mathcal{C}$ since $\varphi_\bullet$ is an isomorphism in $K^{\tiny\mbox{ex}}_{n\mbox{-}\Sigma}(\iota(\mathcal{C}))$. By the Comparison Theorem we have $\Omega^{-n}h\cdot\beta_M=\beta_N\cdot\Sigma h$ in $\underline{\mbox{mod}}\,\mathcal{C}$. Note that we also have $\Omega^{-n}h\cdot\alpha_M=\alpha_N\cdot\Sigma h$ by the naturality of $\alpha$.
 Thus $$\begin{array}{ll}
           & X_\bullet\in\Theta\Leftrightarrow \beta_M=\alpha_M\\
          \Leftrightarrow & (\Omega^{-n}h)^{-1}\cdot\beta_N\cdot\Sigma h=(\Omega^{-n}h)^{-1}\cdot\alpha_N\cdot\Sigma h \\
         \Leftrightarrow  &  \beta_N=\alpha_N\Leftrightarrow Y_\bullet\in\Theta.
        \end{array}
 $$
\end{proof}

\begin{cor} (see \cite[Lemma 2.2]{[BJT]})
Let $(\mathcal{C},\Sigma,\Theta)$ be a pre-$n$-angulated category. Then $\Theta$ must include all contractible $n$-$\Sigma$-sequences.
\end{cor}

\begin{proof}
We note that trivial $n$-$\Sigma$-sequences and contractible $n$-$\Sigma$-sequences are $n$-$\Sigma$-homotopy equivalent, because they are zero objects in $K^{\tiny\mbox{ex}}_{n\mbox{-}\Sigma}(\iota(\mathcal{C}))$. Since each trivial $n$-$\Sigma$-sequence belongs to $\Theta$, by Lemma \ref{2.2} we infer that each contractible $n$-$\Sigma$-sequence belongs to $\Theta$ too.
\end{proof}



Now we are ready to prove the following result.

\begin{thm}\label{thm1} Let $\Theta$ be a full subcategory of $C^{\tiny\mbox{ex}}_{n\mbox{-}\Sigma}(\iota(\mathcal{C}))$ which is closed under translation functor and $n$-$\Sigma$-homotopy equivalence.  Then the following statements are equivalent.

(a) The class $\Theta$ is a pre-$n$-angulation of $(\mathcal{C},\Sigma)$.

(b) The functor $Z_1: \Theta\rightarrow \mbox{mod}\,\mathcal{C}$ is full and dense.

(c) There exists an isomorphism $\alpha:(\Sigma,\sigma) \rightarrow(\Omega^{-n},(-1)^n1_{\Omega^{-n-1}})$ of triangle functors such that $\Theta=\Theta_\alpha$.
\end{thm}

\begin{proof}
$(a)\Rightarrow (b).$
For each $M\in\mbox{mod}\,\mathcal{C}$, we fix a projective presentation
$$\mathcal{C}(-,X_1)\xrightarrow{\mathcal{C}(-,f_1)}\mathcal{C}(-,X_2)\rightarrow M\rightarrow 0.$$
By (N1)(c), we get an $n$-angle $$X_\bullet=(X_1\xrightarrow{f_1}X_2\xrightarrow{f_2}\cdots\xrightarrow{f_{n-1}}X_n\xrightarrow{f_n}\Sigma X_1)\in\Theta.$$
It follows from (N2) that $T_M=X_\bullet[2]\in \Theta$ and $Z_1T_M=M$.

For each morphism $h:M\rightarrow N$ in $\mbox{mod}\,\mathcal{C}$, let $T_M=(X_i,f_i)_{i\in\mathbb{Z}}\in\Theta$ and $T_N=(Y_i,g_i)_{i\in\mathbb{Z}}\in\Theta$ such that $Z_1T_M=M$ and $Z_1T_N=N$. Since $Y_1$ and $Y_2$ are projective-injectives, we obtain the following commutative diagram
$$\xymatrix{
0\ar[r] & M\ar[r]\ar[d]^{h}& X_1 \ar[r]^{f_1}\ar[d]^{\varphi_1} & X_2 \ar[r]^{f_2}\ar[d]^{\varphi_2} & X_3 \ar[r]^{f_3} & \cdots \ar[r]^{f_{n-1}}& X_n \ar[r]^{f_n} & \Sigma X_1 \ar[d]^{\Sigma \varphi_1}\\
0\ar[r] & N\ar[r] & Y_1 \ar[r]^{g_1} & Y_2 \ar[r]^{g_2} & Y_3 \ar[r]^{g_3} & \cdots \ar[r]^{g_{n-1}} & Y_n \ar[r]^{g_n}& \Sigma Y_1\\
}$$ whose rows are exact sequences. By (N3) there exist morphisms $\varphi_i:X_i\rightarrow Y_i$ for $3\leq i\leq n$, such that the above diagram is commutative. We denote by $T(h)=(\varphi_1,\varphi_2,\cdots,\varphi_n)$. Then $Z_1T(h)=h$.

(b) $\Rightarrow$ (c). Given an object $M\in\mbox{mod}\,\mathcal{C}$, assume that $T_M=(X_i,f_i)_{i\in\mathbb{Z}}\in\Theta$ such that $Z_1T_M\cong M$.
 By (3.1), we obtain the following commutative diagram
$$\xymatrix{
0 \ar[r] & M \ar[r]\ar@{=}[d] & X_1 \ar[r]\ar[d] & X_2 \ar[r]\ar[d] & \cdots \ar[r] & X_{n} \ar[r]\ar[d] & \Sigma M \ar[d]^{\alpha_M} \ar[r]& 0\\
0 \ar[r] & M \ar[r] & I_M \ar[r] & I_{\Omega^{-1}M} \ar[r] & \cdots \ar[r] & I_{\Omega^{1-n}M} \ar[r] & \Omega^{-n}M \ar[r] & 0\\
}$$
with exact rows. The  Comparison Theorem implies that $\alpha_M:\Sigma M\rightarrow\Omega^{-n}M$ is an isomorphism in $\underline{\mbox{mod}}\mathcal{C}$. For each morphism $h:M\rightarrow N$ in $\mbox{mod}\,\mathcal{C}$, we deduce that $\Omega^{-n}h\cdot\alpha_M=\alpha_{N}\cdot\Sigma h$ in $\underline{\mbox{mod}}\,\mathcal{C}$ by the Comparison Theorem. Thus we obtain a functorial isomorphism $\alpha:\Sigma\rightarrow\Omega^{-n}$.
For each $M\in\mbox{mod}\,\mathcal{C}$, we may assume that $X_1=I_M$ and $0\rightarrow M\xrightarrow{a_1} X_1\xrightarrow{b_1} \Omega^{-1}M\rightarrow 0$ is an exact sequence. By Lemma \ref{lem2}(b), we have $Z_1(T_M[1])\cong\Omega^{-1}M$. By the construction of the functor $\alpha$, we obtain the following commutative diagram
$$\xymatrix{
0\ar[r] & \Sigma M \ar[r]^{\Sigma a_1}\ar[d]^{\alpha_M} & \Sigma X_1 \ar[r]^{(-1)^{n}\Sigma b_1\ \ }\ar[d] & \Sigma\Omega^{-1}M \ar[d]^{\alpha_{\Omega^{-1} M}}\ar[r]& 0 \\
0\ar[r] & \Omega^{-n} M \ar[r] & I_{\Omega^{-n} M}  \ar[r] & \Omega^{-n-1} M \ar[r] & 0\\
}$$
with exact rows. Since $\Sigma X_1=I_{\Sigma M}$,  we have the following commutative diagram
$$\xymatrix{
0\ar[r] &\Sigma M \ar[r]^{\Sigma a_1}\ar[d]^{\alpha_M} & \Sigma X_1 \ar[r]^{\Sigma b_1}\ar[d] & \Omega^{-1}\Sigma M \ar[d]^{\Omega^{-1}\alpha_{ M}}\ar[r] & 0\\
0\ar[r] &\Omega^{-n} M \ar[r] & I_{\Omega^{-n} M}  \ar[r] & \Omega^{-n-1} M\ar[r] & 0\\
}$$with exact rows. Therefore, $\alpha_{\Omega^{-1}M}=(-1)^{n}\Omega^{-1}\alpha_M$ in $\underline{\mbox{mod}}\,\mathcal{C}$. In other words, $\alpha: (\Sigma,\sigma) \rightarrow (\Omega^{-n},(-1)^n1_{\Omega^{-n-1}})$ is an isomorphism of triangle functors. Now $(\mathcal{C},\Sigma)$ admits a pre-$n$-angulation $\Theta_\alpha$ by Lemma \ref{2.3}.

It remains to show that $\Theta=\Theta_\alpha$. By the construction of $\alpha$,
we have $\Theta\subseteq \Theta_\alpha$. Let $X_\bullet=(X_i, f_i)_{i\in\mathbb{Z}}$ be an object in $\Theta_\alpha$ and $M$ be the kernel of $f_1$. Suppose $T_M=(Y_i,g_i)_{i\in\mathbb{Z}}\in\Theta$ such that $Z_1T_M=M$. Since $Y_1$ and $Y_2$ are projective-injective, we  find morphisms $\varphi_1:X_1\rightarrow Y_1$ and $\varphi_2:X_2\rightarrow Y_2$ such that the following diagram
$$\xymatrix{
& X_1 \ar[r]^{f_1}\ar[dd]^{\varphi_1} & X_2 \ar[r]^{f_2}\ar[dd]^{\varphi_2} & X_3 \ar[r]^{f_3} & \cdots \ar[r]^{f_{n-1}}& X_n \ar[r]^{f_n} & \Sigma X_1 \ar[dd]^{\Sigma \varphi_1}\\
 M \ar[ru]\ar@{=}[dd] & & & & & & & \\
 &  Y_1 \ar[r]^{} & Y_2 \ar[r]^{} & Y_3 \ar[r]^{} & \cdots \ar[r]^{} & Y_n \ar[r]^{}& \Sigma Y_1\\
 M\ar[ru] & & & & & & & & &\\
}$$ commutes. We complete $(\varphi_1,\varphi_2)$ to an $n$-$\Sigma$-periodic morphism $\varphi_\bullet=(\varphi_1,\varphi_2, \cdots,\varphi_n)$ from $X_\bullet$ to $T_M$ by (N3). Since $Z_1(\varphi_\bullet)=1$, we obtain that $\varphi_\bullet: X_\bullet\rightarrow T_M$ is an $n$-$\Sigma$-homotopy equivalence by Lemma \ref{lem2}(a). Thus $X_\bullet\in\Theta$ because $\Theta$ is closed under $n$-$\Sigma$-homotopy equivalence. Therefore, $\Theta=\Theta_\alpha$.

(c) $\Rightarrow$ (a). It follows from  Lemma \ref{2.3}.
\end{proof}



\subsection{$n$-angulations} Keep the notation as above. The functor $Z_1: \Theta\rightarrow \mbox{mod}\,\mathcal{C}$ is called {\em``strongly" full} if for each morphism $h: Z_1X_\bullet\rightarrow Z_1Y_\bullet$ in $\mbox{mod}\,\mathcal{C}$, where $X_\bullet,Y_\bullet\in\Theta$, there exists a morphism $T(h): X_\bullet\rightarrow Y_\bullet$ in $\Theta$ such that $Z_1(T(h))=h$ and the mapping cone $C(T(h))$ belongs to $\Theta$.

 The main theorem of this paper is the following.

\begin{thm}\label{thm2} Let $\Theta$ be a full subcategory of $C^{\tiny\mbox{ex}}_{n\mbox{-}\Sigma}(\iota(\mathcal{C}))$ which is closed under translation functor and $n$-$\Sigma$-homotopy equivalence.  Then the following statements are equivalent.

(a) The class $\Theta$ is an $n$-angulation of $(\mathcal{C},\Sigma)$.

(b) The functor $Z_1: \Theta\rightarrow \mbox{mod}\,\mathcal{C}$ is dense and $\Theta$ satisfies axiom (N4).

(c) The functor $Z_1: \Theta\rightarrow \mbox{mod}\,\mathcal{C}$ is dense and ``strongly" full.
\end{thm}

\begin{proof}
We note that if $\Theta$ satisfies axiom (N4), then it is easy to see that the functor $Z_1: \Theta\rightarrow \mbox{mod}\,\mathcal{C}$ is full. Thus
(a) $\Leftrightarrow$ (b) and (b) $\Rightarrow$ (c) follows from  Theorem \ref{thm1}.

(c) $\Rightarrow$ (b). Given a commutative diagram
$$\xymatrix{
X_1 \ar[r]^{f_1}\ar[d]^{\varphi_1} & X_2 \ar[r]^{f_2}\ar[d]^{\varphi_2} & X_3 \ar[r]^{f_3} & \cdots \ar[r]^{f_{n-1}}& X_n \ar[r]^{f_n} & \Sigma X_1 \ar[d]^{\Sigma \varphi_1}\\
Y_1 \ar[r]^{g_1} & Y_2 \ar[r]^{g_2} & Y_3 \ar[r]^{g_3} & \cdots \ar[r]^{g_{n-1}} & Y_n \ar[r]^{g_n}& \Sigma Y_1\\
}$$ with rows in $\Theta$. Assume that $k_1: M\rightarrow X_1$ is the kernel of $f_1$, $k_1':N\rightarrow Y_1$ is the kernel of $g_1$ and $h:M\rightarrow N$ is the induced morphism. By assumption, there is an object $T(h)=(\phi_1,\phi_2,\cdots,\phi_n)\in\Theta$ such that $Z_1T(h)=h$ and the mapping cone $C(T(h))\in\Theta$. Since $(\varphi_1-\phi_1)k_1=0$, there exists a morphism $h_1:X_2\rightarrow Y_1$ such that $\varphi_1-\phi_1=h_1f_1$. Note that $(\varphi_2-\phi_2-g_1h_1)f_1=0$. There exists a morphism $h_2:X_3\rightarrow Y_2$ such that $\varphi_2-\phi_2=g_1h_1+h_2f_2$. We define $\varphi_3=\phi_3+g_2h_2$, then $g_3\varphi_3=g_3\phi_3=\phi_4f_3$. If we take $\varphi_4=\phi_4,\cdots,\varphi_n=\phi_n$, then $g_i\varphi_i=\varphi_{i+1}f_i$ for $4\leq i\leq n-1$, and $g_n\varphi_{n}=g_n\phi_n=\Sigma\phi_1\cdot f_n=\Sigma(\varphi_1-h_1f_1)\cdot f_n=\Sigma\varphi_1\cdot f_n$. Thus $\varphi_\bullet=(\varphi_1,\varphi_2,\varphi_3,\cdots,\varphi_n)$ is an $n$-$\Sigma$-periodic morphism and $\varphi_\bullet$ is $n$-$\Sigma$-homotopic to $T(h)$ with the $n$-$\Sigma$-homotopy $(h_1,h_2,0,\cdots,0)$. Hence the mapping cone $C(\varphi_\bullet)$ is isomorphic to the mapping cone $C(T(h))$ in $K^{\tiny\mbox{ex}}_{n\mbox{-}\Sigma}(\iota(\mathcal{C}))$. Consequently, $C(\varphi_\bullet)\in\Theta$.
\end{proof}

\begin{cor}\label{cor1}
If there exists a triangle functor $T: \underline{\mbox{mod}}\,\mathcal{C} \rightarrow K^{\tiny\mbox{ex}}_{n\mbox{-}\Sigma}(\iota(\mathcal{C}))$  such that $Z_1T=Id$, then $(\mathcal{C},\Sigma, \Theta)$ is an $n$-angulated category, where $\Theta=T(\underline{\mbox{mod}}\,\mathcal{C})$.
\end{cor}

\begin{proof}
Since $Z_1T=Id$, the functor $Z_1:\Theta\rightarrow \mbox{mod}\,\mathcal{C}$ is dense and full.  For each morphism $h: M\rightarrow N$ in $\mbox{mod}\,\mathcal{C}$, assume that $M\xrightarrow{h} N \rightarrow L\rightarrow \Omega^{-1}M$ is a triangle in $\underline{\mbox{mod}}\,\mathcal{C}$. Then
$TM\xrightarrow{T(h)} TN \rightarrow TL\rightarrow (TM)[1]$ is a triangle in $K^{\tiny\mbox{ex}}_{n\mbox{-}\Sigma}(\iota(\mathcal{C}))$ since $T$ is a triangle functor. It follows that the mapping cone $C(T(h))\cong TL$ in $K^{\tiny\mbox{ex}}_{n\mbox{-}\Sigma}(\iota(\mathcal{C}))$. Thus $C(T(h))\in\Theta$. Now the corollary holds by Theorem \ref{thm2}.
\end{proof}

The following corollary is a higher version of \cite[Theorem 8.1]{[Am]}. It provides a unified way to construct $n$-$\Sigma$-injective resolutions.
We will give some applications in Section 5.2.

\begin{cor}\label{cor2} 
Assume there exists an exact sequence of exact endofunctors of $\mbox{mod}\,\mathcal{C}$
$$\begin{gathered}0\rightarrow \mbox{Id}\rightarrow  X^{1}\rightarrow X^{2}\rightarrow\cdots\rightarrow X^{n}\rightarrow \Sigma\rightarrow 0\end{gathered}\eqno{(3.2)} $$
where  all the $X^{i}$ take values in $\iota(\mathcal{C})$. Then $(\mathcal{C},\Sigma)$ admits an $n$-angulation.
\end{cor}

\begin{proof}
For each $M\in\mbox{mod}\,\mathcal{C}$, we denote by $T_M$ the following $n$-$\Sigma$-periodic exact complex
$$X^1M\rightarrow X^2M\rightarrow\cdots\rightarrow X^nM\rightarrow \Sigma X^1M$$
induced by the  exact sequence (3.2). Then $Z_1T_M=M$. 
Since all $X^i$ are exact functors, it is easy to see that the functor $T:\mbox{mod}\,\mathcal{C} \rightarrow C^{\tiny\mbox{ex}}_{n\mbox{-}\Sigma}(\iota(\mathcal{C}))$, which sends an object $M$ to $T_M$, is an exact functor preserving the projective-injectives. Thus $T$ induces a triangle functor $T: \underline{\mbox{mod}}\mathcal{C} \rightarrow K^{\tiny\mbox{ex}}_{n\mbox{-}\Sigma}(\iota(\mathcal{C}))$ such that $Z_1T=Id$. By Corollary \ref{cor1}, we are done.
\end{proof}

\begin{rem}\label{rem4.2}
Let $\mathcal{C}$ be an additive category and $\Sigma$ be an automorphism of $\mathcal{C}$. By Theorem \ref{thm2}, we can construct an $n$-angulation as follows. For each $M\in\mbox{mod}\,\mathcal{C}$, we fix an $n$-$\Sigma$-periodic injective resolution $T_M\in C^{\tiny\mbox{ex}}_{n\mbox{-}\Sigma}(\iota(\mathcal{C}))$.
Denote by $\Theta$ the full subcategory  of $ C^{\tiny\mbox{ex}}_{n\mbox{-}\Sigma}(\iota(\mathcal{C}))$ consisting of objects $X_\bullet$ which is $n$-$\Sigma$-homotopy equivalent to some $T_M$. If $\Theta$ is not closed under the translation functor, then $\Theta$ is not an $n$-angulation of $(\mathcal{C},\Sigma)$. Otherwise, $\Theta$ is an  $n$-angulation if $\Theta$ moreover satisfies (N4). More details and examples will be given in Section 4 and Section 5.
\end{rem}

\section{First examples}

In this section, we will apply the ideas of Section 3 to unify the $n$-angulated structure of known examples including algebraic triangulated categories, the standard construction of $n$-angulated categories and the $n$-angulated categories from local rings. We don't plan to prove that they are $n$-angulated categories. Instead, we only want to understand the $n$-angulated structure with a new point of view.

\subsection{Algebraic triangulated categories}
Let $(\mathcal{B},\mathcal{S})$ be a Frobenius category. We denote by $\underline{\mathcal{B}}$ the stable category $\mathcal{B}/[\mathcal{I}]$, where $\mathcal{I}$ is the full subcategory consisting of projective-injectives. Given a morphism $f:X\rightarrow Y$ in $\mathcal{B}$, we denote by $\underline{f}$ the image of $f$ in $\underline{\mathcal{B}}$ under the canonical functor $\mathcal{B}\rightarrow\underline{\mathcal{B}}$. By \cite[Proposition 2.2]{[Ha]}, there exists an automorphism $\Sigma:\underline{\mathcal{B}}\rightarrow\underline{\mathcal{B}}$. We recall the definition of $\Sigma$ as follows. For each $X\in \mathcal{B}$, we fix a short exact sequence $0\rightarrow X\xrightarrow{i_X} I(X)\xrightarrow{p_X} \Sigma X\rightarrow 0$ in $\mathcal{S}$ such that $I(X)\in\mathcal{I}$. For each morphism $f:X\rightarrow Y$, we get the following commutative diagram
$$\xymatrix{
0\ar[r] & X\ar[r]^{i_X}\ar[d]^{f} & I(X)\ar[r]^{p_X}\ar[d]^{I(f)} & \Sigma X \ar[d]^{\Sigma f}\ar[r] & 0\\
0\ar[r] & Y\ar[r]^{i_Y} & I(Y)\ar[r]^{p_Y} & \Sigma Y \ar[r] & 0\\
}$$ with rows in $\mathcal{S}$.
 It is easily seen that $\underline{\Sigma f}$ does not depend on the choice of $I(f)$. We define $\Sigma \underline{f}=\underline{\Sigma f}$.

\begin{lem}\label{5-0}
Let $$0\rightarrow X_1\oplus I_1\xrightarrow{\left(
                              \begin{smallmatrix}
                               f_{11} &  f_{12} \\
                                 f_{13} &  f_{14} \\
                              \end{smallmatrix}
                            \right)} X_2\oplus I_2\xrightarrow{\left(
                              \begin{smallmatrix}
                               f_{21} &  f_{22} \\
                              \end{smallmatrix}
                            \right)} X_3 \rightarrow0$$ be a short exact sequence in $\mathcal{S}$, where $I_1,I_2\in\mathcal{I}$.
Then the sequence
 $$\underline{\mathcal{B}}(-,X_1)\xrightarrow{\underline{\mathcal{B}}(-,f_{11})} \underline{\mathcal{B}}(-,X_2)\xrightarrow{\underline{\mathcal{B}}(-,f_{21})} \underline{\mathcal{B}}(-,X_3)$$ is exact.
\end{lem}

\begin{proof}
Since $f_{21}f_{11}+f_{22}f_{13}=0$, we have $\underline{f_{21}f_{11}}=0$. Consequently, $\mbox{Im}\,\underline{\mathcal{B}}(-,f_{11})\subseteq \mbox{ker} \,\underline{\mathcal{B}}(-,f_{21})$. Assume that $\underline{g}:Y\rightarrow X_2$ is a morphism in $\mbox{ker}\,\underline{\mathcal{B}}(Y,f_{21})$, then we have $\underline{f_{21}g}=0$.
There exist two morphisms $a:Y\rightarrow I$ and $b:I\rightarrow X_3$ for some $I\in\mathcal{I}$ such that $f_{21}g=ba$. Since $I$ is  projective-injective, we can assume that
$b=(f_{21},f_{22})\left(
                              \begin{smallmatrix}
                              x_1 \\
                              x_2\\
                              \end{smallmatrix}
                            \right)$.
We note that $(f_{21},f_{22})\left(
                              \begin{smallmatrix}
                              g-x_1a \\
                             -x_2a\\
                              \end{smallmatrix}
                            \right)=0$. There exists a morphism $\left(
                              \begin{smallmatrix}
                              h_1 \\
                             h_2\\
                              \end{smallmatrix}
                            \right):Y\rightarrow X_1\oplus I_1$ such that $\left(
                              \begin{smallmatrix}
                              g-x_1a \\
                             -x_2a\\
                              \end{smallmatrix}
                            \right)=\left(
                              \begin{smallmatrix}
                               f_{11} &  f_{12} \\
                                 f_{13} &  f_{14} \\
                              \end{smallmatrix}
                            \right)\left(
                              \begin{smallmatrix}
                              h_1 \\
                             h_2\\
                              \end{smallmatrix}
                            \right)$. Thus $\underline{g}=\underline{h_1f_{11}}\in\mbox{Im}\,\underline{\mathcal{B}}(Y,f_{11})$.
Therefore, $\mbox{ker}\,\underline{\mathcal{B}}(Y,f_{21})=\mbox{Im}\,\underline{\mathcal{B}}(Y,f_{11})$.
\end{proof}

\begin{lem}\label{5-1}
Assume that the following diagram
$$\begin{gathered}\xymatrix{
0\ar[r] & X_1\ar[r]^{i_1}\ar[d]^{f_1} & I(X_1)\ar[r]^{p_1}\ar[d]^{a_1} & \Sigma X_1\ar@{=}[d]\ar[r] & 0\\
0\ar[r] & X_2\ar[r]^{f_2} & X_3\ar[r]^{f_3} & \Sigma X_1 \ar[r] & 0\\
}\end{gathered}\eqno(4.1)$$
is commutative with rows in $\mathcal{S}$.
Then the  complex
 $$\begin{gathered} X_1\xrightarrow{\underline{f_1}} X_2\xrightarrow{\underline{f_2}} X_3\xrightarrow{\underline{f_3}} \Sigma X_1 \end{gathered}\eqno(4.2)$$
 belongs to $C^{\mbox{\tiny ex}}_{3\mbox{-}\Sigma}({\iota}(\underline{\mathcal{B}}))$.
\end{lem}

\begin{proof} We check the lemma without using Lemma \ref{1.1}.
By the definition of $\Sigma$, we have the following commutative diagram
 $$\xymatrix{
0\ar[r] & X_1\ar[r]^{i_1}\ar[d]^{f_1} & I(X_1)\ar[r]^{p_1}\ar[d]^{I(f_1)} & \Sigma X_1\ar[d]^{\Sigma f_1}\ar[r] & 0\\
0\ar[r] & X_2\ar[r]^{i_2} & I(X_2)\ar[r]^{p_2} & \Sigma X_2 \ar[r] & 0\\
}$$ with rows in $\mathcal{S}$. Since the left square in diagram (4.1) is a pushout, there exists a unique morphism $a_2: X_3\rightarrow I(X_2)$ such that $i_2=a_2f_2$ and $I(f_1)=a_2a_1$. Noting that  $$(p_2a_2)f_2=p_2i_2=0=(\Sigma f_1\cdot f_3)f_2,$$
 $$(p_2a_2)a_1=p_2I(f_1)=\Sigma f_1\cdot p_1=(\Sigma f_1\cdot f_3)a_1,$$
 we obtain $p_2a_2=\Sigma f_1\cdot f_3$ by the universal property of pushout.
Thus we have the following commutative diagram
$$\xymatrix{
0\ar[r] & X_1\ar@{=}[d]\ar[r]^{\left(
                              \begin{smallmatrix}
                               f_{1} \\
                               i_{1} \\
                              \end{smallmatrix}
                            \right)\ \ \ \ } & X_2\oplus I(X_1) \ar[d]^{(0,1)}\ar[r]^{\ \ \ \ (f_2,-a_1)} & X_3\ar[r]\ar[d]^{-f_3} & 0\\
0\ar[r] & X_1\ar[r]^{i_1}\ar[d]^{f_1} & I(X_1)\ar[r]^{p_1}\ar[d]^{a_1} & \Sigma X_1\ar@{=}[d]\ar[r] & 0\\
0\ar[r] & X_2\ar[r]^{f_2}\ar@{=}[d] & X_3\ar[r]^{f_3}\ar[d]^{a_2} & \Sigma X_1 \ar[r]\ar[d]^{\Sigma f_1} & 0\\
0\ar[r] & X_2\ar[r]^{i_2} & I(X_2)\ar[r]^{p_2} & \Sigma X_2\ar[r] & 0\\
}$$
with rows in $\mathcal{S}$.

Therefore, we have three short exact sequences as follows:
$$0\rightarrow X_1\xrightarrow{\left(
                              \begin{smallmatrix}
                             f_1 \\
                             i_1\\
                              \end{smallmatrix}
                            \right)}X_2\oplus I(X_1)\xrightarrow{(f_2,-a_1)} X_3\rightarrow 0,$$
$$0\rightarrow X_2\oplus I(X_1)\xrightarrow{\left(
                              \begin{smallmatrix}
                             f_2 & -a_1 \\
                             0 & 1\\
                              \end{smallmatrix}
                            \right)} X_3\oplus I(X_1)\xrightarrow{(f_3,p_1)} \Sigma X_1\rightarrow 0,$$
 $$0\rightarrow X_3\xrightarrow{\left(
                              \begin{smallmatrix}
                             f_3 \\
                             a_2\\
                              \end{smallmatrix}
                            \right)}\Sigma X_1 \oplus I(X_2)\xrightarrow{(\Sigma f_1,-p_2)} \Sigma X_2\rightarrow 0.$$
 Now the lemma follows from Lemma \ref{5-0}.
\end{proof}

Denote by $\Delta$ the collection of 3-$\Sigma$-sequences in $C^{{\tiny\mbox{ex}}}_{3\mbox{-}\Sigma}({\iota}(\underline{\mathcal{B}}))$ which are isomorphic to the ones in the form of (4.2).

\begin{prop}\label{5.3}
The functor $Z_1:\Delta\rightarrow\mbox{mod}\,\mathcal{\underline{B}}$ is dense.
\end{prop}

\begin{proof} We check the result without using Theorem \ref{thm2}.
For each $M\in\mbox{mod}\,\mathcal{C}$, we choose a projective presentation
$$\underline{\mathcal{B}}(-,X_1)\xrightarrow{\underline{\mathcal{B}}(-,f_1)}\underline{\mathcal{B}}(-,X_2)\rightarrow M\rightarrow 0.$$
Then we have a commutative diagram (4.1). By Lemma \ref{5-1}, we get
$$X_\bullet=(X_1\xrightarrow{\underline{f_1}} X_2\xrightarrow{\underline{f_2}} X_3\xrightarrow{\underline{f_3}} \Sigma X_1)\in\Delta.$$
Note that we have a commutative diagram
$$\xymatrix{
0\ar[r] & X_2\ar[r]^{i_2}\ar[d]^{f_2} & I(X_2)\ar[r]^{p_2}\ar[d]^{\left(
                              \begin{smallmatrix}
                            0 \\
                             1\\
                              \end{smallmatrix}
                            \right)} & \Sigma X_2\ar@{=}[d]\ar[r] & 0\\
0\ar[r] & X_3\ar[r]^{\left(
                              \begin{smallmatrix}
                             f_3 \\
                             a_2\\
                              \end{smallmatrix}
                            \right)\ \ \ \ } & \Sigma X_1\oplus I(X_2)\ar[r]^{\ \ \ \ (-\Sigma f_1,p_2)} & \Sigma X_2 \ar[r] & 0\\
}$$
with rows in $\mathcal{S}$. Lemma \ref{5-1} implies that $X_\bullet[1]\in\Delta$. Therefore, $X_\bullet[2]\in\Delta$ and $Z_1(X_\bullet[2])=M$.
\end{proof}


\begin{example} (\cite[Theorem 2.6]{[Ha]}) Let $(\mathcal{B},\mathcal{S})$ be a Frobenius category, then $(\underline{\mathcal{B}},\Sigma,\Delta)$ is a triangulated category.
\end{example}

\subsection{Standard construction of $n$-angulated categories}
In this subsection, we assume that  $\mathcal{T}$ is a triangulated category with the suspension functor $\Sigma$ and $\mathcal{C}$ is a full subcategory of $\mathcal{T}$. Suppose  $d$ is a positive integer. Recall that the subcategory $\mathcal{C}$ is called $d$-$rigid$ if $\mathcal{T}(\mathcal{C},\Sigma^i\mathcal{C})=0$ for $i=1,2,\cdots,d-1$.
The subcategory $\mathcal{C}$ is called $d$-$cluster\ tilting$ if $\mathcal{C}$ is functorially finite and
$$\left.
    \begin{array}{rl}
      \mathcal{C}& =\{X\in\mathcal{T}|\mathcal{T}(X,\Sigma^i\mathcal{C})=0, \forall i=1,\cdots,d-1\} \\
       &  =\{X\in\mathcal{T}|\mathcal{T}(\mathcal{C},\Sigma^iX)=0, \forall i=1,\cdots,d-1\}.\\
    \end{array}
  \right.
$$
Let $\mathcal{X}$, $\mathcal{Y}$ and $\mathcal{Z}$ be full subcategories of $\mathcal{T}$. We denote by $\mathcal{X}\ast\mathcal{Y}$ the class of objects $T$ in $\mathcal{T}$ such that there is a triangle $X\rightarrow T\rightarrow Y\rightarrow \Sigma X$ with $X\in\mathcal{X}$ and $Y\in\mathcal{Y}$. The octahedral axiom implies that $(\mathcal{X}\ast\mathcal{Y})\ast\mathcal{Z}=\mathcal{X\ast}(\mathcal{Y}\ast\mathcal{Z})$.

\begin{rem}\label{rem5}
(a) Let $\mathcal{C}$ be a $d$-rigid subcategory which is closed under $\Sigma^d$, then $\mathcal{T}(\mathcal{C},\Sigma^i\mathcal{C})=0$ for each $i\notin d\mathbb{Z}$.

(b) The subcategory $\mathcal{C}$ is $d$-cluster tilting if and only if $\mathcal{C}$ is  $d$-rigid  and
$\mathcal{T}=\mathcal{C}\ast\Sigma\mathcal{C}\ast\Sigma^2\mathcal{C}\ast\cdots\ast\Sigma^{d-1}\mathcal{C}$.
\end{rem}

\begin{proof}
(a) It is clear. (b) It is proved in \cite{[ZZ]}.
\end{proof}

\begin{lem}\label{5.1}
Let $\mathcal{C}$ be a $d$-rigid subcategory which is closed under $\Sigma^d$. Suppose that there are the following triangles:
$$X_1\xrightarrow{\alpha_1} X_2\xrightarrow{\alpha_2'} X_{2.5}\xrightarrow{\alpha_{d+2}^1} \Sigma X_1, $$
$$X_{i.5}\xrightarrow{\alpha_{i}''} X_{i+1}\xrightarrow{\alpha_{i+1}'} X_{i+1.5}\xrightarrow{\alpha_{d+2}^{i}} \Sigma X_{i.5},$$
$$X_{d.5}\xrightarrow{\alpha_d''} X_{d+1}\xrightarrow{\alpha_{d+1}} X_{d+2}\xrightarrow{\alpha_{d+2}^d} \Sigma X_{d.5},$$
where $2\leq i\leq d-1$, $X_j\in\mathcal{C}$ for $j\in\mathbb{Z}$ and $X_k\in\mathcal{T}$ for $k\notin\mathbb{Z}$.
Then
$$\begin{gathered}X_1\xrightarrow{\alpha_1} X_2\xrightarrow{\alpha_2} \cdots\xrightarrow{\alpha_{d+1}} X_{d+2}\xrightarrow{\alpha_{d+2}} \Sigma^d X_1\end{gathered}\eqno{(4.3)} $$
belongs to $C^{\tiny\mbox{ex}}_{(d+2)\mbox{-}\Sigma^d}(\iota(\mathcal{C}))$, where $\alpha_i=\alpha_i''\alpha_i'$ for $2\leq i\leq d$ and $$\alpha_{d+2}=\Sigma^{d-1}\alpha_{d+2}^1\cdot\Sigma^{d-2}\alpha_{d+2}^2\cdots\Sigma\alpha_{d+2}^{d-1}\cdot\alpha_{d+2}^d.$$
\end{lem}

\begin{proof}
For each object $A\in\mathcal{C}$, applying the functor $\mathcal{T}(A,-)$ to the above triangles, we obtain the following exact sequences $$0\rightarrow\mathcal{T}(A,\Sigma^{-1} X_{2.5})\rightarrow\mathcal{T}(A, X_{1})\rightarrow\mathcal{T}(A, X_{2})\rightarrow\mathcal{T}(A, X_{2.5})\rightarrow 0,$$
$$0\rightarrow\mathcal{T}(A,\Sigma^{-1} X_{i+1.5})\rightarrow\mathcal{T}(A, X_{i.5})\rightarrow\mathcal{T}(A, X_{i+1})
\rightarrow\mathcal{T}(A, X_{i+1.5})\rightarrow\mathcal{T}(A,\Sigma X_{i.5})\rightarrow 0,$$
$$0\rightarrow\mathcal{T}(A,X_{d.5})\rightarrow\mathcal{T}(A, X_{d+1})\rightarrow\mathcal{T}(A, X_{d+2})
\rightarrow\mathcal{T}(A,\Sigma X_{d.5})\rightarrow 0,$$
where $2\leq i\leq d-1$. Moreover, for $1\leq i\leq d-2$ and $2\leq j\leq d-1$, we have
 $\mathcal{T}(A,\Sigma^iX_{2.5})=0$, $\mathcal{T}(A,\Sigma^jX_{d.5})=0$ and
$\mathcal{T}(A,\Sigma^i\alpha_{d+2}^j): \mathcal{T}(A,\Sigma^iX_{j+1.5})\rightarrow\mathcal{T}(A, \Sigma^{i+1}X_{j.5})$ is an isomorphism.
Since $\mathcal{T}(A,\Sigma X_{i.5})\cong\mathcal{T}(A,\Sigma^{i-1} X_{2.5})=0$ and
$\mathcal{T}(A,\Sigma^{-1} X_{i+1.5})\cong\mathcal{T}(\Sigma^dA,\Sigma^{d-1} X_{i+1.5})\cong\mathcal{T}(\Sigma^dA,\Sigma^{i} X_{d.5})=0$
 for $2\leq i\leq d-1$, gluing the above exact sequences, we have the following long exact sequence
 $$\mathcal{T}(A,X_1)\rightarrow\mathcal{T}(A,X_2)\rightarrow\cdots\rightarrow\mathcal{T}(A,X_{d+2})\rightarrow\mathcal{T}(A,\Sigma^d X_1)\rightarrow\mathcal{T}(A,\Sigma^d X_2),$$
 where the last but one morphism is the composition
 $$\mathcal{T}(A, X_{d+2})\twoheadrightarrow\mathcal{T}(A,\Sigma X_{d.5})\xrightarrow{\simeq}\mathcal{T}(A,\Sigma^2 X_{d-1.5})\xrightarrow{\simeq}\cdots\xrightarrow{\simeq}\mathcal{T}(A,\Sigma^{d-1} X_{2.5})\rightarrowtail\mathcal{T}(A,\Sigma^d X_{1}).$$
\end{proof}

Denote by $\Theta$ the class of $(d+2)$-$\Sigma^d$-sequences in $C^{\tiny\mbox{ex}}_{(d+2)\mbox{-}\Sigma^d}(\iota(\mathcal{C}))$ of the form (4.3).

\begin{prop}\label{5.2}
Let $\mathcal{C}$ be a $d$-cluster tilting subcategory which is closed under $\Sigma^d$, then the functor $Z_1:\Theta\rightarrow \mbox{mod}\,\mathcal{C}$ is dense.
\end{prop}

\begin{proof} We prove the proposition without using Theorem \ref{thm2}.
For each $M\in\mbox{mod}\,\mathcal{C}$, there exists an exact sequence $\mathcal{C}(-,X_1)\xrightarrow{\mathcal{C}(-,f_1)}\mathcal{C}(-,X_2)\rightarrow M\rightarrow 0$. Assume that $X_1\xrightarrow{f_1} X_2\rightarrow X_{2.5}\rightarrow\Sigma X_1$ is a triangle. Then it is easy to see that
$M\cong\mathcal{T}(-, X_{2.5})|_{\mathcal{C}}$.
It follows from Remark \ref{rem5} that $X_{2.5}\in\mathcal{T}=\mathcal{C}\ast\Sigma\mathcal{C}\ast\Sigma^2\mathcal{C}\ast\cdots\ast\Sigma^{d-1}\mathcal{C}$.
Thus we have the following triangles:
$$X_{i.5}\rightarrow X_{i+1}\rightarrow X_{i+1.5}\rightarrow \Sigma X_{i.5},$$
$$X_{d.5}\rightarrow X_{d+1}\rightarrow X_{d+2}\rightarrow \Sigma X_{d+1.5},$$
where $2\leq i\leq d-1$. By Lemma \ref{5.1}, we have a complex $X_\bullet=(X_1\xrightarrow{f_1} X_2\rightarrow\cdots\rightarrow X_{d+2}\rightarrow \Sigma^dX_1)\in\Theta$.
It is easy to see that $T_{M}=X_\bullet[2]\in\Theta$ and $Z_1T_M=M$.
\end{proof}

\begin{example} (\cite[Theorem 1]{[GKO]})
Let $\mathcal{T}$ be a triangulated category with the suspension functor $\Sigma$ and $\mathcal{C}$ be  a $d$-cluster tilting subcategory such that $\Sigma^d\mathcal{C}\subseteq\mathcal{C}$, then $(\mathcal{C},\Sigma^d,\Theta)$ is a $(d+2)$-angulated category.
\end{example}

\subsection{$n$-angulated categories from local rings}
Let $R$ be a commutative local ring with a principle maximal ideal $\mathfrak{m}=(p)$ such that $\mathfrak{m}^2=0$. Then $\mathfrak{m}$ is the unique nontrivial ideal of $R$.  By the Baer Criterion, it is easy to check that $R$ is a selfinjective ring. Moreover, we have $\mbox{mod} R=\mbox{add}\,(R\oplus \mathfrak{m})$, where $\mbox{mod} R$ is the category of finitely generated $R$-modules.  We denote by $\mathcal{C}$ the category of finitely generated projective $R$-modules and by $\Sigma$ the identity functor of $\mathcal{C}$.

\begin{lem}\label{lr1}
Each minimal $n$-$\Sigma$-periodic injective resolution of $\mathfrak{m}$ is in the form of
$$ R(u)_\bullet=( R\xrightarrow{up}R\xrightarrow{p}\cdots\xrightarrow{p}R\xrightarrow{p} \Sigma R)$$
for some unit $u$ in $R$. Moreover, $R(u)_\bullet$ is $n$-$\Sigma$-homotopy equivalent to $R(v)_\bullet$ if and only if $up=vp$.
\end{lem}

\begin{proof}
It is clear that $R(u)_\bullet$ is a minimal $n$-$\Sigma$-periodic injective resolution of $\mathfrak{m}$ for each unit $u$ in $R$. Conversely, let
$T_\mathfrak{m}$ be a minimal $n$-$\Sigma$-periodic injective resolution of $\mathfrak{m}$. We can assume that
$$T_\mathfrak{m}=(R\xrightarrow{u_1p}R\xrightarrow{u_2p}\cdots\xrightarrow{u_{n-1}p}R\xrightarrow{u_np} \Sigma R)$$
where all $u_i$ are units in $R$. It is obvious that $T_\mathfrak{m}$ is isomorphic to $$ R(u)_\bullet= (R\xrightarrow{up}R\xrightarrow{p}\cdots\xrightarrow{p}R\xrightarrow{p} \Sigma R)$$
where $u=u_1u_2\cdots u_n$.
Assume that $R(u)_\bullet$ is $n$-$\Sigma$-homotopy equivalent to $R(v)_\bullet$, then we have the following commutative diagram
$$\xymatrix{
R\ar[r]^{up}\ar[d]^{w_1} & R\ar[r]^{p}\ar[d]^{w_2} & \cdots \ar[r]^{p} &  R\ar[r]^{p}\ar[d]^{w_n} & \Sigma R \ar[d]^{\Sigma w_1}\\
R\ar[r]^{vp}\ar[d]^{w'_1} & R\ar[r]^{p}\ar[d]^{w'_2} & \cdots \ar[r]^{p} &  R\ar[r]^{p}\ar[d]^{w'_n} & \Sigma R \ar[d]^{\Sigma w'_1}\\
R\ar[r]^{up} & R\ar[r]^{p} & \cdots \ar[r]^{p} &  R\ar[r]^{p} & \Sigma R. \\
}$$ Moreover, there exist morphisms $h_i:R\rightarrow R$ for $1\leq i\leq n$  such that $w'_1w_1-1=h_1up+ph_n$ and $w'_2w_2-1=h_2p+uph_1$.
Noting that $p=w'_1w_1p=w_2'w_2p$, $uw_1'p=vw_2'p$ and $w_2p=w_3p=\cdots=w_np=w_1p$, we have $up=uw_1'w_1p=vw_2'w_1p=vw_2'w_2p=vp$. The ``if" part is trivial.
\end{proof}

For each unit $u$ in $R$, we denote by $\Theta_u$ the class of $n$-$\Sigma$-sequences in $C^{\tiny\mbox{ex}}_{n\mbox{-}\Sigma}(\iota(\mathcal{C}))$ consisting of $F_\bullet \oplus A_\bullet$, where $F_\bullet$ is a finite direct sum of $R(u)_\bullet$ and $A_\bullet$ is a contractible $n$-$\Sigma$-sequence.

\begin{prop}
The functor $Z_1:\Theta_u\rightarrow \mbox{mod}\,R$ is dense.
\end{prop}

\begin{proof}
Since $\mbox{mod} R=\mbox{add}\,(R\oplus \mathfrak{m})$, $Z_1R(u)_\bullet=\mathfrak{m}$ and $Z_1A_\bullet=R$, where $A_\bullet=(R\rightarrow0\rightarrow\cdots\rightarrow 0\rightarrow \Sigma R\xrightarrow{1} \Sigma R)$, the functor $Z_1$ is dense.
\end{proof}

\begin{example}\label{ex1.2} (\cite[Theorem 3.6]{[BJT]})
For each unit $u$ in $R$, the triple $(\mathcal{C},\Sigma,\Theta_u)$ is an $n$-angulated category whenever $n$ is even, or when $n$ is odd and $2p=0$.
\end{example}

\begin{rem}\label{lr2}
 If $n$ is odd and $2p\neq0$, then by Lemma \ref{lr1}, $R(-u)_\bullet\notin\Theta_u$ for each unit $u$ in $R$. But $R(u)_\bullet[1]\cong R(-u)_\bullet$, thus $(\mathcal{C},\Sigma,\Theta_u)$ is not an $n$-angulated category since $\Theta_u$ is not closed under the translation functor. In fact, $(\mathcal{C},\Sigma)$ does not admit any $n$-angulation in this case.
\end{rem}

\section{More examples}

In this section, we will first discuss the $n$-angulated structure of semisimple categories, then apply Corollary \ref{cor2} to get a class of new examples of $n$-angulated categories from quasi-periodic selfinjective algebras. At last, we construct other $n$-angulations from known ones.

\subsection{Semisimple categories}

We recall that an additive category $\mathcal{C}$ is $semisimple$ if each morphism $f:X\rightarrow Y$ in $\mathcal{C}$ factors as $f=hg$, where $g$ is a split epimorphism and $h$ is a split monomorphism.

\begin{rem}\label{rem}
Let $\mathcal{\mathcal{C}}$ be a semisimple category. Then $\mathcal{C}$ is idempotent complete. Moreover,
for each morphism $f:X\rightarrow Y$ in $\mathcal{\mathcal{C}}$, there exists a commutative diagram
 $$\xymatrix{
 X\ar[r]^{f}\ar[d]^\alpha & Y \ar[d]^{\beta}\\
 A\oplus B \ar[r]^{\left(
                                                              \begin{smallmatrix}
                                                                0 & 0 \\
                                                                0 &  1 \\
                                                              \end{smallmatrix}
                                                            \right)} &
 C\oplus B\\
 }$$ where $\alpha$ and $\beta$ are isomorphisms.
\end{rem}

We recall the following description of semisimple categories given by Jasso.

\begin{lem}\label{lem3}(\cite[Theorem 3.9]{[J]})
Let $m$ be a positive integer. Then the $m$-abelian categories in which every $m$-exact sequence is contractible are precisely the semisimple categories.
\end{lem}

 We can compare Lemma \ref{lem3} with the following.

\begin{thm} \label{thm0}
Let $\mathcal{C}$ be an additive category and $\Sigma$ an automorphism of $\mathcal{C}$.
Then $(\mathcal{C}, \Sigma)$ is an $n$-angulated category in which each $n$-angle is contractible if and only if $\mathcal{C}$ is semisimple.

\end{thm}

\begin{proof}
Assume that $\mathcal{C}$ is semisimple. Denote by $\Theta$ the class of all contractible $n$-$\Sigma$-sequences. We claim that $(\mathcal{C},\Sigma,\Theta)$ is an $n$-angulated category. Indeed, it is easily seen that $\Theta$ is closed under isomorphisms, direct sums, direct summands and rotations. For each morphism $f:X\rightarrow Y$ in $\mathcal{C}$, by Remark \ref{rem} we can assume that $f$ is the morphism of the form $\left(
                                                              \begin{smallmatrix}
                                                                0 & 0 \\
                                                                0 & 1 \\
                                                              \end{smallmatrix}
                                                            \right):
A\oplus B\rightarrow C\oplus B$. Thus there exists the following contractible $n$-$\Sigma$-sequence
$$A\oplus B\xrightarrow{f} C\oplus B\xrightarrow{(1,0)} C\rightarrow 0\rightarrow\cdots\rightarrow 0\rightarrow \Sigma A\xrightarrow{\left(
                                                                                                                                       \begin{smallmatrix}
                                                                                                                                         1 \\
                                                                                                                                         0 \\
                                                                                                                                       \end{smallmatrix}
                                                                                                                                     \right)
} \Sigma A\oplus\Sigma B.$$
Since $\mathcal{C}$ is idempotent complete, each contractible $n$-$\Sigma$-sequence is a direct sum of trivial $n$-angles.  It is easy to see that (N3) holds.
Let $\varphi_\bullet:X_\bullet\rightarrow Y_\bullet$ be a morphism of contractible $n$-$\Sigma$-sequences. Since the mapping cone $C(\varphi_\bullet)$ is isomorphic to $X_\bullet[1]\oplus Y_\bullet$, axiom (N4) holds.

Conversely, assume that $\mathcal{C}$ is an $n$-angulated category in which each $n$-angle is contractible.  Given a morphism $f_1: X_1\rightarrow X_2$ in $\mathcal{C}$,
suppose that $X_\bullet=(X_1\xrightarrow{f_1} X_2\xrightarrow{f_2}\cdots\xrightarrow{f_{n-1}} X_n\xrightarrow{f_n} \Sigma X_1)$ is an $n$-angle. Since $X_\bullet\in C^{\tiny\mbox{ex}}_{n\mbox{-}\Sigma}(\iota(\mathcal{C}))$ is contractible, $\mathcal{C}(-,f_1)$ has a factorization
$$\xymatrix{
\mathcal{C}(-,X_1)\ar[rr]^{\mathcal{C}(-,f_1)}\ar@{->>}[rd]^{a_1}& & \mathcal{C}(-,X_2)\\
& M \ar@{>->}[ru]^{b_1} & \\
}$$ for some $M\in\mbox{mod}\,\mathcal{C}$ such that $a_1$ is a split epimorphism and $b_1$ is a split monomorphism (see \cite[Section 2.3]{[GKO]}).
It follows from Theorem \ref{thm2} that  $\mbox{mod}\,\mathcal{C}$ = $\mbox{proj}\,\mathcal{C}$. So we can assume that $M=\mathcal{C}(-,Y)$, $a_1=\mathcal{C}(-,f_1')$ and $b_1=\mathcal{C}(-,f_1'')$. The Yoneda Lemma shows that $f_1=f_1''f_1'$ where $f_1'$ is a split epimorphism  and $f_1''$ is a split monomorphism. Therefore, $\mathcal{C}$ is a semisimple category by definition.
\end{proof}

The following result characterizes semisimple categories in terms of $n$-angulated categories and $m$-abelian categories.

\begin{cor}\label{cor3}
Let $\mathcal{C}$ be an additive category, $m$ and $n$ be positive integers where $n\geq3$. Then $\mathcal{C}$ is an $n$-angulated and $m$-abelian category if and only if $\mathcal{C}$ is a semisimple category.
\end{cor}

\begin{proof} The ``if" part follows from Theorem \ref{thm0} and Lemma \ref{lem3}. For the ``only if" part,
we take an $n$-exact sequence
$\begin{gathered} X_0\xrightarrow{f_0}X_1\xrightarrow{f_1}\cdots\xrightarrow{f_{n-1}}X_{n}\xrightarrow{f_n} X_{n+1}. \end{gathered}$
Since $f_0$ is a monomorphism, we infer that $f_0$ is a split monomorphism by Lemma \ref{lem1.2}. Thus all the $n$-exact sequences are contractible by \cite[Proposition 2.6]{[J]}. Lemma \ref{lem3} implies that $\mathcal{C}$ is a semisimple category.
\end{proof}

\subsection{Quasi-periodic selfinjective algebras}

Let $k$ be a field and $A$ be a finite-dimensional $k$-algebra. Given an algebra automorphism $\sigma$ of $A$, we denote by $_1A_\sigma$ the bimodule structure on $A$ where the right action is twisted by $\sigma$. It is easy to check that, given two automorphisms $\sigma$ and $\tau$, there is an isomorphism $_1A_\sigma\otimes_A(_1A_\tau)\cong$ $_1A_{\tau\sigma}$.
A finite-dimensional $k$-algebra $A$ is said to be $quasi$-$periodic$ if  $A$ has a quasi-periodic projective resolution over the enveloping algebra $A^e=A^{\tiny\mbox{op}}\otimes_k A$, i.e., $\Omega^n_{A^e}(A)\cong$ $_1A_\sigma$ as $A$-$A$-bimodule for some natural number $n$ and some automorphism $\sigma$ of $A$. In particular, $A$ is  $periodic$ if $\Omega^n_{A^e}(A)\cong A$.

In this subsection, we assume that $A$ is a finite-dimensional indecomposable $k$-algebra such that $\Omega^n_{A^e}(A)\cong$ $_1A_\sigma$ as an $A$-$A$-bimodule for some automorphism $\sigma$ of $A$ and for some $n\geq3$. Since an indecomposable quasi-periodic algebra $A$ is self-injective by \cite[Lemma 1.5]{[GSS]}, the category $\mbox{mod}\,A$ is a Frobenius category. Denote by proj$A$ the category of finitely generated projective modules over $A$. Note that $e_iA\otimes_A(_1A_{\sigma^{-1}})\cong e_iA_{\sigma^{-1}}\cong\sigma(e_i)A$ for each idempotent $e_i$ of $A$, hence the functor $-\otimes_A(_1A_{\sigma^{-1}}): \mbox{proj}A\rightarrow\mbox{proj}A$ is an automorphism.

\begin{thm}\label{4.2}
The category $(\mbox{proj}\,A, -\otimes_A(_1A_{\sigma^{-1}}))$ admits an $n$-angulation.
\end{thm}

\begin{proof} Since $\Omega^n_{A^e}(A)\cong$ $_1A_\sigma$, there exists an exact sequence of $A$-$A$-bimodules
$$\begin{gathered} 0\rightarrow\ _1A_\sigma \rightarrow P_n\rightarrow P_{n-1}\rightarrow \cdots\rightarrow P_1\rightarrow A\rightarrow 0
\end{gathered}\eqno{(5.1)} $$
where the $P_i$'s  are projective as bimodules. Tensoring this sequence with $_1A_{\sigma^{-1}}$ yields the following exact sequence of $A$-$A$-bimodules
$$
0\rightarrow A \rightarrow\ _1A_{\sigma^{-1}}\otimes_A P_n\rightarrow\ _1A_{\sigma^{-1}}\otimes_AP_{n-1}\rightarrow \cdots\rightarrow\ _1A_{\sigma^{-1}}\otimes_AP_1\rightarrow\ _1A_{\sigma^{-1}} \rightarrow 0
 $$ where all the $_1A_{\sigma^{-1}}\otimes_AP_i$ are projective as bimodules. Thus we have the following exact sequence of exact endofunctors of mod$A$
$$0\rightarrow Id \rightarrow -\otimes_A(_1A_{\sigma^{-1}})\otimes_A P_n\rightarrow  -\otimes_A(_1A_{\sigma^{-1}})\otimes_AP_{n-1}\rightarrow \cdots$$
$$\cdots\rightarrow -\otimes_A(_1A_{\sigma^{-1}})\otimes_AP_1\rightarrow  -\otimes_A(_1A_{\sigma^{-1}}) \rightarrow 0.$$
Moreover, the functors $-\otimes_A(_1A_{\sigma^{-1}})\otimes_AP_i$ take values in proj$A$.
 By Corollary \ref{cor2}, $(\mbox{proj}A, -\otimes_A(_1A_{\sigma^{-1}}))$ admits an $n$-angulation.
\end{proof}

\begin{cor}
For each positive integer $m$,  the category $(\mbox{proj}A, -\otimes_A(_1A_{\sigma^{-m}}))$ admits an $mn$-angulation.
In particular, if $\sigma$ is of finite order $l$, then $(\mbox{proj}A, \mbox{Id}_{\mbox{proj}A})$ admits an $ln$-angulation.
\end{cor}

\begin{proof} Since $\Omega^n_{A^e}(A)\cong$ $_1A_\sigma$, we claim that $\Omega^{mn}_{A^e}(A)\cong$ $_1A_{\sigma^m}$, thus the corollary follows from Theorem \ref{4.2}. Indeed, we only need to show that there exists an exact sequence of $A$-$A$-bimodules
$$\begin{gathered}
 0\rightarrow\ _1A_{\sigma^{m}} \rightarrow P_{mn}\rightarrow P_{mn-1}\rightarrow \cdots\rightarrow P_{(m-1)n+1}\rightarrow\cdots\rightarrow P_{n}\rightarrow\cdots\rightarrow P_1\rightarrow A\rightarrow 0
\end{gathered}$$
where the $P_i$'s  are projective as bimodules.  We will prove it by induction on $m$.
It is trivial for $m=1$.
Assume that $m>1$ and our claim holds for $m-1$.
Applying the functor $_{1}A_{\sigma^{m-1}}\otimes_A-$ to the  sequence (5.1), we obtain the following exact sequence of $A$-$A$-bimodules
$$0\rightarrow\ _{1}A_{\sigma^m}\rightarrow\ _{1}A_{\sigma^{m-1}}\otimes_AP_n\rightarrow\cdots\rightarrow\  _{1}A_{\sigma^{m-1}}\otimes_AP_1\rightarrow\ _{1}A_{\sigma^{m-1}}\rightarrow 0. \ \eqno (5.2)$$
We take $P_{(m-1)n+i}=\ _{1}A_{\sigma^{m-1}}\otimes_AP_i$, where $i=1,2,\cdots,n$. Then the $P_{(m-1)n+i}$'s are projective as bimodules. By induction and the sequence (5.2), our claim holds for each positive integer $m$.
\end{proof}


\subsection{Global automorphisms}

Given an $n$-angulation $\Theta$, we want to construct more $n$-angulations from known ones. For this,  global automorphisms of $(\mathcal{C}, \Sigma)$, which were  introduced by Balmer, play an important role.

\begin{defn} (\cite{[Bal]})
Let $\mathcal{C}$ be an additive category and $\Sigma$ be an automorphism of $\mathcal{C}$. A $global\ automorphism$ $\alpha$ of $(\mathcal{C}, \Sigma)$ is an invertible endomorphism of the identity functor $Id:\mathcal{C}\rightarrow\mathcal{C}$ which commutes with $\Sigma$. In other words, a global automorphism $\alpha$ is a collection of isomorphisms $\alpha_A: A\rightarrow A$, for all objects $A$ in $\mathcal{C}$, such that $\alpha_B f=f\alpha_A$ for each morphism $f: A\rightarrow B$ and $\alpha_{\Sigma A}=\Sigma\alpha_A$ for each object $A$ in $\mathcal{C}$.
\end{defn}

\begin{example}\label{ex}
(a) Let $k$ be a field, $\mathcal{C}$ be a $k$-linear category and $\Sigma$ be a $k$-linear automorphism of $\mathcal{C}$. Then for each $\lambda\in k\setminus\{0\}$,
 $\alpha=\{\alpha_A=\lambda 1_A: A\rightarrow A|A\in\mathcal{C}\}$ is a global automorphism of $(\mathcal{C}, \Sigma)$.

 (b) Let $R$ be a commutative ring, $\mathcal{C}$ be the category of finitely generated projective modules over $R$ and $\Sigma$ be the identity functor of $\mathcal{C}$. Then for each unit $u$ in $R$,  multiplication by $u$ gives a global automorphism $\lambda_u$ of $(\mathcal{C}, \Sigma)$.
\end{example}

The following is a higher version of \cite[Proposition 4]{[Bal]}.

\begin{prop}\label{prop}
Let $\Theta$ be an $n$-angulation of $(\mathcal{C},\Sigma)$ and $\alpha$ be a global automorphism of $(\mathcal{C},\Sigma)$. Define $\Theta^\alpha$  as the class of $n$-$\Sigma$-sequences $$X_\bullet=(X_1\xrightarrow{f_1}X_2\xrightarrow{f_2}X_3\xrightarrow{f_3}\cdots\xrightarrow{f_{n-1}}X_n\xrightarrow{f_n}\Sigma X_1)$$ such that
$$X^\alpha_\bullet=(X_1\xrightarrow{f_1\alpha_{X_1}}X_2\xrightarrow{f_2}X_3\xrightarrow{f_3}\cdots\xrightarrow{f_{n-1}}X_n\xrightarrow{f_n}\Sigma X_1)\in\Theta.$$ Then $\Theta^\alpha$ is an $n$-angulation of $(\mathcal{C},\Sigma)$.
\end{prop}

\begin{proof}
For each $M\in \mbox{mod}\,\mathcal{C}$, $X_\bullet$ is an $n$-$\Sigma$-periodic injective resolution of $M$ if and only if  so is $X^\alpha_\bullet$  since $\alpha_{X_1}$ is an isomorphism. By Theorem \ref{thm2}, we only need to show that $\Theta^\alpha$ satisfies (N4).

Given a commutative diagram
$$\xymatrix{
X_1 \ar[r]^{f_1}\ar[d]^{\varphi_1} & X_2 \ar[r]^{f_2}\ar[d]^{\varphi_2} & X_3 \ar[r]^{f_3} & \cdots \ar[r]^{f_{n-1}}& X_n \ar[r]^{f_n} & \Sigma X_1 \ar[d]^{\Sigma \varphi_1}\\
Y_1 \ar[r]^{g_1} & Y_2 \ar[r]^{g_2} & Y_3 \ar[r]^{g_3} & \cdots \ar[r]^{g_{n-1}} & Y_n \ar[r]^{g_n}& \Sigma Y_1\\
}$$ with rows in $\Theta^\alpha$, we have the following commutative diagram
$$\xymatrix{
X_1 \ar[r]^{f_1\alpha_{X_1}}\ar[d]^{\varphi_1} & X_2 \ar[r]^{f_2}\ar[d]^{\varphi_2} & X_3 \ar[r]^{f_3}\ar@{-->}[d]^{\varphi_3} & \cdots \ar[r]^{f_{n-1}}& X_n \ar[r]^{f_n} \ar@{-->}[d]^{\varphi_n}& \Sigma X_1 \ar[d]^{\Sigma \varphi_1}\\
Y_1 \ar[r]^{g_1\alpha_{Y_1}} & Y_2 \ar[r]^{g_2} & Y_3 \ar[r]^{g_3} & \cdots \ar[r]^{g_{n-1}} & Y_n \ar[r]^{g_n}& \Sigma Y_1\\
}$$ with rows in $\Theta$, which can be completed to a whole commutative diagram such that the mapping cone $C(\varphi_\bullet)$ belongs to $\Theta$ by (N4).
The following commutative diagram
$$\xymatrixcolsep{3.5pc}\xymatrixrowsep{3.5pc}\xymatrix{
X_2\oplus Y_1\ar[r]^{\left(
             \begin{smallmatrix}
               -f_2 & 0\\
                \varphi_2 & g_1\alpha_{Y_1}\\
             \end{smallmatrix}
           \right)}\ar[d]^{\left(
             \begin{smallmatrix}
               \alpha_{X_2}^{-1} & 0\\
               0 & 1\\
             \end{smallmatrix}
           \right)} & X_3\oplus Y_2\ar[r]^{\left(
             \begin{smallmatrix}
               -f_3 & 0 \\
               \varphi_3 & g_2\\
             \end{smallmatrix}
           \right)}\ar@{=}[d] & \cdots \ar[r]^{\left(
             \begin{smallmatrix}
               -f_n & 0 \\
               \varphi_n & g_{n-1}\\
             \end{smallmatrix}
           \right)} & \Sigma X_1\oplus Y_n\ar[r]^{\left(
             \begin{smallmatrix}
               -\Sigma(f_1\alpha_{X_1}) & 0 \\
               \Sigma\varphi_1 & g_n\\
             \end{smallmatrix}
           \right)}\ar@{=}[d] & \Sigma X_2\oplus\Sigma Y_1\ar[d]^{\left(
             \begin{smallmatrix}
              \Sigma\alpha_{X_2}^{-1} & 0 \\
               0 & 1\\
             \end{smallmatrix}
           \right)}\\
X_2\oplus Y_1\ar[r]^{\left(
             \begin{smallmatrix}
               -f_2\alpha_{X_2} & 0\\
                \varphi_2\alpha_{X_2} & g_1\alpha_{Y_1}\\
             \end{smallmatrix}
           \right)} & X_3\oplus Y_2\ar[r]^{\left(
             \begin{smallmatrix}
               -f_3 & 0 \\
               \varphi_3 & g_2\\
             \end{smallmatrix}
           \right)} & \cdots \ar[r]^{\left(
             \begin{smallmatrix}
             - f_n & 0\\
               \varphi_n& g_{n-1}\\
             \end{smallmatrix}
           \right)} & \Sigma X_1\oplus Y_n\ar[r]^{\left(
             \begin{smallmatrix}
               -\Sigma f_1 & 0 \\
               \Sigma\varphi_1 & g_n\\
             \end{smallmatrix}
           \right)} & \Sigma X_2\oplus\Sigma Y_1\\ }$$
implies that the second row belongs to $\Theta$ since the first row $C(\varphi_\bullet)\in\Theta$. Noting that
$$\left(
             \begin{smallmatrix}
               -f_2\alpha_{X_2} & 0\\
                \varphi_2\alpha_{X_2} & g_1\alpha_{Y_1}\\
             \end{smallmatrix}
           \right)=\left(
             \begin{smallmatrix}
               -f_2 & 0\\
                \varphi_2 & g_1\\
             \end{smallmatrix}
           \right)\left(
             \begin{smallmatrix}
              \alpha_{X_2} & 0\\
                0 &\alpha_{Y_1}\\
             \end{smallmatrix}
           \right)=\left(
             \begin{smallmatrix}
               -f_2 & 0\\
                \varphi_2 & g_1\\
             \end{smallmatrix}
           \right)\alpha_{X_2\oplus Y_1},$$
 we obtain that $$X_2\oplus Y_1\xrightarrow{\left(
                              \begin{smallmatrix}
                                -f_2 & 0 \\
                                \varphi_2 & g_1 \\
                              \end{smallmatrix}
                            \right)}
 X_3\oplus Y_2 \xrightarrow{\left(
                              \begin{smallmatrix}
                                -f_3 & 0 \\
                                \varphi_3 & g_2 \\
                              \end{smallmatrix}
                            \right)}
 \cdots \xrightarrow{\left(
                            \begin{smallmatrix}
                               -f_n & 0 \\
                                \varphi_n & g_{n-1} \\
                             \end{smallmatrix}
                           \right)}
 \Sigma X_1\oplus Y_n \xrightarrow{\left(
                              \begin{smallmatrix}
                                -\Sigma f_1 & 0 \\
                                \Sigma\varphi_1 & g_n \\
                              \end{smallmatrix}
                            \right)}
 \Sigma X_2\oplus \Sigma Y_1 \\$$ belongs to $\Theta^\alpha$.
\end{proof}

\begin{rem} We denote by G-Aut$(\mathcal{C},\Sigma)$ the set of all global automorphisms of $(\mathcal{C}, \Sigma)$. It is easy to see that G-Aut$(\mathcal{C},\Sigma)$ is a group under composition of natural transformations.
Proposition \ref{prop} implies that the group G-Aut$(\mathcal{C},\Sigma)$ acts on the set of $n$-angulations of $(\mathcal{C},\Sigma)$ from the right. In general,  the action is not free, but we do not know whether the action is transitive or not.
\end{rem}

\begin{example}
In Example \ref{ex1.2}, for each  $n$-angulation $\Theta$ of $(\mathcal{C},\Sigma)$, by Lemma \ref{lr1}, it is not hard to show that $\Theta=\Theta_u$ for some unit $u$ in $R$. Moreover, $\Theta_u=\Theta_v$ if and only if $up=vp$. By Example \ref{ex}(b) and Proposition \ref{prop}, we have $\Theta_u=\Theta_1^{\lambda_u}$, where 1 is the identity element of $R$ and $\lambda_u$ is the global automorphism induced by $u$. Thus in this case the group of global automorphism of $(\mathcal{C},\Sigma)$ acts transitively on the set of all $n$-angulations. But the action is not free.
\end{example}

\vspace{2mm}\noindent {\bf Acknowledgements} \ The author thanks Xiaowu Chen for drawing his attention to the reference \cite{[Am]} and for valuable conversations on this topic. Some part of the work was done when the author visited University of Stuttgart. He wants to thank Steffen Koenig for warm hospitality during his stay in Stuttgart and for helpful discussions and remarks. He also thanks  Wei Hu for helpful discussions.

\end{document}